\documentclass[11pt,a4paper]{article}

\usepackage{amsmath, amsfonts, amssymb,color,csquotes}
\usepackage{hyperref}
 
\hypersetup{%
	colorlinks=true, linkcolor=blue,
	citecolor=blue
}

\def\be#1\ee{\begin{equation}#1\end{equation}}

\newtheorem{thm}{Theorem}[section]
\newtheorem{lem}[thm]{Lemma}
\newtheorem{prop}[thm]{Proposition}

\newtheorem{rem}[thm]{Remark}
\newtheorem{defn}[thm]{Definition}

\DeclareMathOperator{\Var}{Var}

\def\P{{\mathbb{P}}}
\def\R{\mathbb{R}}
\def\E{\mathbb{E}\,}
\def\Q{\mathbb{Q}}

\def\N{{\mathbb N}}

\newenvironment{proof}[1][] {\noindent {\bf Proof#1:} }{\hspace*{\fill}$\square$\medskip\par}

\def\Beta{\text{{\rm Beta}}}

\def\cov{\textrm{cov}}
\def\corr{\textrm{corr}}

\newcommand{\eps}{\varepsilon}
\newcommand{\fa}{\textcolor{black}}
\newcommand{\minorchanges}{\textcolor{black}}

\def\NN{{\mathcal N}}

\def \=L{\ {\buildrel\hbox{\scriptsize d }\over =}\ }

\let\BFseries\bfseries\def\bfseries{\BFseries\mathversion{bold}} 

\begin{document}

\title{Persistence probabilities of weighted sums of stationary Gaussian sequences}

\author{Frank Aurzada  and Sumit Mukherjee\thanks{Research partially supported by NSF grant DMS-1712037}\\
Technical University of Darmstadt and Columbia University}

\maketitle

\begin{abstract}
With $\{\xi_i\}_{i\ge 0}$ being a centered stationary Gaussian sequence with non-negative correlation function $\rho(i):=\E[ \xi_0\xi_i]$ and  $\{\sigma(i)\}_{i\ge 1}$ a sequence of positive reals, we study the asymptotics of the persistence probability of the weighted sum $\sum_{i=1}^\ell \sigma(i) \xi_i$, $\ell\ge 1$. For summable correlations $\rho$, we show that the persistence exponent is universal.
On the contrary, for non-summable $\rho$, even for polynomial weight functions $\sigma(i)\sim i^p$ the persistence exponent depends on the rate of decay of the correlations (encoded by a parameter $H$) and on the polynomial rate $p$ of $\sigma$. In this case, we show existence of the persistence exponent $\theta(H,p)$ and study its properties as a function of $(p,H)$. During the course of our proofs, we develop several tools for dealing with exit problems for Gaussian processes with non-negative correlations -- e.g.\ a continuity result for persistence exponents and a necessary and sufficient criterion for the persistence exponent to be zero -- that might be of independent interest. 
\end{abstract}


\footnotesize{{\bf Keywords:}  first passage time;
Fractional Brownian Motion; Gaussian process; 

persistence probability; stationary process.}

\footnotesize{{\bf AMS 2020 subject classification:} Primary: 60G22, 60G15; Secondary: 60G10}

\section{Introduction}
The study of the tail behaviour of the first passage time of a random walk $\sum_{i=1}^n X_i$ with independent and identically distributed increments $\{X_i\}$ above (or below) a level $x\in\R$,
$$
\P( \max_{1\le k\le n} \sum_{i=1}^k X_i < x ), \qquad \text{as $n\to\infty$},
$$
is a classical topic in probability theory. This type of problems is studied both for discrete and continuous time \cite{bertoin,doney1,doney2,feller,sato}; and due to its fundamental nature it has numerous applications in finance, insurance, queueing, and other subjects. In a recent work, Denisov, Sakhanenko and Wachtel \cite{DSV} study the case when the increments $\{X_i\}$ are independent, but not necessarily identically distributed. When the random variables have finite variance, Denisov et al.\ show that the properly time-rescaled version of the corresponding random walk (with non-identically distributed increments) has the same tail behaviour of the first passage time as the classical (i.i.d.) random walk. A particular case is when $X_n=\sigma(n) \xi_n$ with i.i.d.\ $\{\xi_n\}$ with $\E \xi_1 = 0$, where one has (cf.\ Theorem~2 in \cite{DSV})
\begin{equation} \label{eqn:dsv-inde}
\P( \max_{1\le k\le n} \sum_{i=1}^k \sigma(i)\xi_i < x ) \sim \sqrt{\frac{2}{\pi}} \, \frac{x}{s(n)},
\end{equation}
where $s^2(n):=\sum_{i=1}^n \sigma^2(n)$.

The purpose of the present paper is to study the corresponding result in the case of {\em correlated} Gaussian random variables. The picture is much more diverse, and as we shall see, the type of behaviour of the first passage time (\ref{eqn:dsv-inde}) strongly depends on the correlations as well as the weights.

To fix the notation for this paper, let $\{\xi_i\}_{i\ge 0}$ be a centered stationary Gaussian sequence with $\E \xi_i^2 = 1$ and  non-negative correlation function $\rho(i):=\E \xi_0\xi_i$. Let $\{\sigma(i)\}_{i\ge 1}$ be a sequence of positive real numbers.
Define a centered Gaussian process $\{S_\ell\}_{\ell\ge 0}$ by setting 
$$
S_\ell:=\sum_{i=1}^\ell \sigma(i)\xi_i,\qquad \ell\ge 0.
$$

In this paper, we are interested in studying the asymptotics of the persistence probability for the sequence $\{S_\ell\}_{\ell\ge 0}$, defined by
$$
q_n:=\P(\max_{1\le \ell\le n}S_\ell<0),\qquad \text{as $n\to\infty$},
$$
for a wide class of choices of $\sigma(.)$ and $\rho(.)$. 

We will find that for summable $\rho$ the persistence probability is universally determined only by $\sigma$ via the function $s^2(n):=\sum_{i=1}^n \minorchanges{\sigma^2(i)}$ and the behaviour resembles the independent case (\ref{eqn:dsv-inde}). Contrary to this, for non-summable correlations $\rho$, the picture is significantly richer. Here, we study the case of polynomial weights $\sigma(i)\sim i^p$ and polynomial correlations $\rho(i)\sim \kappa i^{2H-2}$ with $H\in(1/2,1)$ and $H+p>1/2$. In this case, we show the existence of the persistence exponent $\theta(H,p)$ and its limiting behaviour when $p$ or/and $H$ approach the natural boundaries of their ranges.

Let us briefly comment on the related literature. In the case of i.i.d.\ random walks, there is a huge literature on first passage times, both classical (see already  \cite{feller}) and recent (see \cite{AS} for a review). In the independent, but non-identically distributed case, an early work is \cite{AB}, before \cite{DSV} gave an essentially complete solution to the problem. In the correlated but unweighted case ($\sigma\equiv 1$), one has to mention the works \cite{abuck,agpp,luysivakoff}. To the knowledge of the authors, no results seem available in the literature when the increments are correlated and weighted random variables.
For general background on persistence problems and their significance in theoretical physics, we refer to the surveys \cite{majumdar2,majumdar1,SCDC} (for a theoretical physics point of view) and 
\cite{AS} (for a review of the mathematical literature).

This paper is structured as follows. In Section~\ref{sec:mainresults}, we state our main results, distinguishing the case of summable and non-summable correlations. Section~\ref{sec:tools} is devoted to a few tools for non-exit problems for stationary Gaussian processes with non-negative correlations that might be of independent interest.  The proof of the universal result in the summable case is given in Section~\ref{sec:proofofsummable}, while the proofs for the results in the non-summable case are given in Section~\ref{sec:proofnonsummable}. Finally, Section~\ref{sec:proofofgeneraltools} contains the proofs of the general tools from  Section~\ref{sec:tools}.
\\
\\

\section{Main results} \label{sec:mainresults}
\subsection{A universal result for summable correlations}
We first study the case when the correlation $\rho(.)$ is summable, i.e.
\begin{align}\label{eq:summable}
\sum_{i=0}^\infty \rho(i)<\infty.
\end{align}
We will also make the following three assumptions on the sequence of weights $\sigma(.)$:
\begin{align}\label{eq:sigma2}
\lim_{n\to\infty}s(n)=\infty,
\end{align}
\begin{align}\label{eq:sigma}
\lim_{n\to\infty}\frac{\sigma(n+\ell)}{\sigma(n)}=1,\qquad \text{for all $\ell\ge 1$},
\end{align}
and 
there exists a $C<\infty$ depending on $\sigma$ such that for all $m\ge n$ we have 
\begin{align}\label{eq:log_concave}
\frac{\sigma(m)}{s(m)}\le C \frac{\sigma(n)}{s(n)}.
\end{align}

Then we can formulate the main result in the summable case.
\begin{thm}\label{thm:summble}
Let $\sigma(.)$ satisfy \eqref{eq:sigma2}, \eqref{eq:sigma}, and \eqref{eq:log_concave}. Then for any correlation function $\rho(.)$ satisfying \eqref{eq:summable} we have
 $$ \lim_{n\to\infty}\frac{\log q_n}{\log s(n)}=-1.$$
\end{thm}

Note that the order of log persistence only depends on $\sigma$ (through the function $s(.)$), and is hence independent of $\rho(.)$. 
The reason for the universality in the above theorem is that the limiting process which governs the exponent always turns out to be the Ornstein-Uhlenbeck process.

Below we give a list of common choices of weight functions for which Theorem \ref{thm:summble} holds. We use the notation $f(i)\sim g(i)$ to denote
$\lim_{i\to\infty} f(i)/g(i)=1$. In all cases, \eqref{eq:sigma2}, \eqref{eq:sigma}, and \eqref{eq:log_concave} can be checked relatively easily, possibly using \cite{BGT}.

\paragraph{Examples.}
{\bf (i)}
Suppose 
$\sigma(i)\sim i^p$
where $p>-1/2$, where one has $s^2(n)\sim n^{2p+1} /(2p+1)$, so that Theorem~\ref{thm:summble} gives
\begin{equation} \label{eqn:savealine02-17}
\lim_{n\to\infty}\frac{\log q_n}{\log n}=-\Big(p+\frac{1}{2}\Big).
\end{equation}

{\bf (ii)}
More generally, suppose $\sigma(.)$ is a regularly varying function of order $p>-1/2$. In this case one shows that $s^2(n)\sim n\sigma^2(n) /(2p+1)$, so that Theorem~\ref{thm:summble} gives again (\ref{eqn:savealine02-17}).

{\bf (iii)}
Suppose $\sigma(i)\sim  i^{-1/2}$ so that $s^2(n)\sim \log n$ and by Theorem~\ref{thm:summble}
$$
\lim_{n\to\infty}\frac{\log q_n}{\log \log n}=-\frac{1}{2}.
$$

{\bf (iv)}
Suppose $\sigma(i)\sim e^{\gamma i^p}$ for some $p\in (0,1)$ and $\gamma>0$. In this case $s^2(n)\sim  n^{1-p}\sigma^2(n) /(2\gamma p)$ and
Theorem~\ref{thm:summble} then gives
$$
\lim_{n\to\infty}\frac{\log q_n}{ n^p}=-\gamma.
$$

We now turn our attention to the assumptions on the weight sequence $\sigma(.)$ made in Theorem \ref{thm:summble}. Our first  proposition shows that without \eqref{eq:sigma2} the conclusion of Theorem \ref{thm:summble} does not hold. In this case the persistence probability does not go to $0$ with $n$, but instead converges to some number between $(0,1)$, which is not universal and depends both on $\rho(.)$ and $\sigma(.)$.

\begin{prop}\label{ppn:easy}
If $\lim_{n\to\infty}s^2(n)=\sum_{i=1}^\infty\minorchanges{\sigma^2(i)}<\infty$ $($i.e. \eqref{eq:sigma2} does not hold$)$, then for  $\rho(.)$ satisfying \eqref{eq:summable} we have $\lim_{n\to\infty}q_n=q$, for some $q\in (0,1)$, where $q$ depends on both $\rho$ and $\sigma$.
\end{prop}

Our second proposition shows that \eqref{eq:sigma} is necessary for Theorem \ref{thm:summble}, as can be seen from a counterexample with exponential weights. In this case the limiting exponent depends on both $\sigma$ and $\rho$.
\begin{prop}\label{ppn:exp}
Suppose $\sigma(i)=e^{\alpha i}$ for some $\alpha>0$, and let $\rho(.)$ satisfy \eqref{eq:summable}. Then we have 
$\lim\limits_{n\to\infty}\frac{\log q_n}{\log n}=-\theta_d(D_{\alpha,\rho})$, where
$$
\theta_d(D_{\alpha,\rho}):=-\lim_{n\to\infty}\frac{1}{n}\log \P(\max_{1\le \ell\le n}Z_\alpha(\ell)<0)\in (0,\infty),
$$
where $\{Z_\alpha(\ell)\}_{\ell\ge 1}$ is a discrete time centered stationary Gaussian sequence  with correlation function 
$$
D_{\alpha,\rho}(\tau):=\frac{\sum_{i,j=0}^\infty e^{-(i+j)\alpha}\rho(i-j-\tau)}{\sum_{i,j=0}^\infty e^{-(i+j)\alpha}\rho(i-j)}.
$$
\end{prop}

\begin{rem}
 It is unclear whether condition \eqref{eq:log_concave} of Theorem \ref{thm:summble}  is actually necessary or whether it is an artifact of our proof technique. We note that assumptions \eqref{eq:sigma2} and \eqref{eq:sigma} already imply that the sequence $\{\frac{\sigma(n)}{s(n)}\}$ converges to $0$; and assumption \eqref{eq:log_concave} demands that this sequence is eventually non-increasing up to a universal constant. Besides being a natural regularity condition, \eqref{eq:log_concave} does hold for a lot of natural choices of weight functions $\sigma$, as can be seen from the above examples.
\end{rem}

\subsection{The non-summable case}
We will now study the case when the correlation function $\rho(.)$ is  not summable.  In particular, we
will assume that there exist $\kappa>0$ and $H\in (1/2,1)$ such that 
\begin{align}
\label{eq:rho}\lim_{i\to\infty}\frac{\rho(i)}{i^{2H-2}}=\kappa.
\end{align}
Note that \eqref{eq:rho} implies that $\rho$ is not summable. In this case, the persistence exponent is no longer governed by the Ornstein-Uhlenbeck process as in Theorem \ref{thm:summble}, but instead heavily depends on the choice of $\rho$ and $\sigma$. We will demonstrate this by assuming that $\sigma$ satisfies
\begin{align}
\label{eq:sigma_p}\lim_{i\to\infty}\frac{\sigma(i)}{i^p}=1,
\end{align}
and showing that the persistence exponent depends on both $p$ and $H$.  To introduce the limiting process, we need the following definition.
\begin{defn} \label{def:smallf}
For $H\in (1/2,1)$,  $p+H>0$ define the function $f_{p,H}:(0,\infty)^2\mapsto (0,\infty)$ by setting $$f_{p,H}(a,b):=\int_0^a \int_0^b x^p y^p |x-y|^{2H-2}dxdy.$$ 
\end{defn}
The function $f_{p,H}(.)$ is well defined for $(p,H)$ with $p+H>0$ because
$$f_{p,H}(a,b)\le f_{p,H}(\max(a,b),\max(a,b))=\max(a,b)^{2p+2H}f_{p,H}(1,1)<\infty,$$
by Selberg's integral formula (cf.\ \cite[(1.2)]{FW}).
\begin{defn}
Define the correlation function of a stationary process on $[0,\infty)$ by setting
\begin{equation}
C_{p,H}(\tau):=e^{-\tau(p+H)}\frac{f_{p,H}(1,e^\tau)}{f_{p,H}(1,1)} .
\label{eqn:covnewprocess}
\end{equation}
\end{defn}
\begin{rem}
For a fractional Brownian Motion $B^H$, define:
$$
Z(\tau) = e^{-(p+H)\tau} \int_0^{e^\tau} x^{p} d B^H(x),\quad \tau\in \R,$$
cf.\ \cite{mishura} for the definition of a stochastic integral w.r.t.\ FBM. Then it can be checked that the function given in \eqref{eqn:covnewprocess} is the correlation function of $Z$. In particular, this shows that $C_{p,H}$ is indeed a correlation function.
\end{rem}

\begin{defn}
Given a non-negative correlation function $A(.)$ on $[0,\infty)$, let $\{Z(t),t\ge 0\}$ be a centered stationary Gaussian process with $A(.)$ as its correlation function, and let $\theta(A)\in [0,\infty]$ be defined as
$$\theta(A):=-\lim_{T\to\infty}\frac{1}{T}\log \P(\sup_{t\in [0,T]}Z(t)<0),$$
where the existence of the limit follows by Slepian's lemma and subadditivity. 
\end{defn}

We can now formulate the second main result of this paper, which handles the non-summable case.

\begin{thm}\label{thm:1}
Let  $\sigma(.)$ and $\rho(.)$ satisfy \eqref{eq:rho} and \eqref{eq:sigma_p}, respectively, for some  $\kappa>0$, and $(p,H)$ such that $ H\in (1/2,1)$, $p+H>0$. Then we have $$\lim_{n\to\infty}\frac{\log q_n}{\log n}=-\theta(C_{p,H}).$$
where $C_{p,H}$ is as in \eqref{eqn:covnewprocess}. Further, the exponent $\theta(p,H):=\theta(C_{p,H})$ lies in $(0,\infty)$.

\end{thm}

We remark that in the case $p=0$, when $S_\ell=\sum_{i=1}^\ell \xi_i$, the results of \cite{abuck,agpp,luysivakoff} imply that $\theta(0,H)=1-H$. The limiting process in this case is the (exponentially time-changed) fractional Brownian motion, for which the exponent was obtained in \cite{molchan1999}.

There seems to be no way to obtain the exponent $\theta(p,H)$ {\it explicitly} for any other $(p,H)$ presently. This is in contrast to the summable correlation case, where for the choice $\sigma(i)\sim  i^p$ the exponent equals $p+1/2$, see (\ref{eqn:savealine02-17}).
Our next theorem explores some properties of the persistence exponent $\theta(p,H)$ as $p$ and $H$ vary. 

\begin{thm}\label{exp:conti}
The exponent $\theta({p,H})$ of Theorem \ref{thm:1} is continuous jointly in $(p,H)$ on the domain $H\in (1/2,1), p+H>0$. It further satisfies 
\begin{eqnarray}
\label{eq:limH}\lim_{H\uparrow 1}\frac{\theta({p,H})}{1-H}&=&1\qquad\text{ for }p>-1,\\
\label{eq:limH2}\lim_{H\downarrow \frac{1}{2}}\theta({p,H})&=&p+\frac{1}{2},\qquad \text{ for }p>-\frac{1}{2},\\
\label{eq:limp}\lim_{p\to\infty}\theta({p,H})&=&\infty,\\
\label{eq:limpinfty} \lim_{p\to\infty}\frac{\theta({p,H})}{p}&=&0,\\
\label{eq:limp+h}\lim_{p\downarrow - H}\frac{\theta({p,H})}{p+H}&=&1.
\end{eqnarray}

\end{thm}

In particular, $\theta({p,H})=p+\frac{1}{2}$ is contradicted by \eqref{eq:limH}, \eqref{eq:limpinfty}, or \eqref{eq:limp+h} as well as $\theta(0,H)=1-H$.

\begin{rem}
It would be interesting to see if one can obtain sharper estimates than the ones provided in \eqref{eq:limp} and \eqref{eq:limpinfty}.  Heuristic calculations suggest that as $p\to\infty$ one has $$p^{2-2H} \lesssim \theta(p,H) \lesssim p^{2-2H}\log p.$$
\end{rem}

\section{General tools} \label{sec:tools}
In this section, we state a few general results on persistence of Gaussian processes which we will apply in the sequel, and some of which may be of independent interest. Almost all our results apply for both discrete time and continuous time Gaussian processes, with time index set $\N:=\{1,2,\ldots\}$ and $\R_{\ge 0}:=[0,\infty)$ respectively. For unifying the statements, we will denote the time index set by $\mathbb{T}$, which is either $\N$ or $\R_{\ge 0}$. Also we will use $\mu$ to denote the counting measure if $\mathbb{T}=\N$, and the Lebesgue measure if $\mathbb{T}=\R_{\ge 0}$. We will also assume throughout that the sample paths of the Gaussian process are continuous almost surely on $\mathbb{T}$, so that the supremum over compact sets in $\mathbb{T}$ is a well defined random variable. If $\mathbb{T}=\N$, then continuity holds vacuously, as any function on $\N$ is continuous.
\\

\fa{Throughout this section, let  $A(.)$ be a non-negative correlation function on $\mathbb{T}$, and let  $\{Z(t),t\in \mathbb{T}\}$ be a centered stationary Gaussian process with correlation function $A(.)$.}

We first state a lemma which gives a necessary and sufficient condition in full generality for truly exponential decay of the persistence probability for stationary Gaussian processes with non-negative correlations. There are sufficient conditions in the literature for truly exponential decay (cf.\ \cite{DM2, FF, FFN, FFJNN}), but in our understanding none of them are both necessary and sufficient. 
\begin{defn}
For any $r\in \R$ let
$$\theta(A,r):=-\lim_{T\to\infty}\frac{1}{T}\log \P(\sup_{t\in [0,T]\cap \mathbb{T} }Z(t)<r),$$
As before, existence of the limit follows by Slepian's lemma and subadditivity.
\end{defn}
\begin{lem}\label{lem:exist}
Assume that $Z$ has continuous sample paths almost surely and that $A(.)$ is non-negative. 
\begin{enumerate}
\item[(a)]
If $\int_\mathbb{T} A(t)\mu(dt)<\infty$, then $\theta(A,r)\in (0,\infty)$ for every $r\in \R$.

\item[(b)]
If $\int_\mathbb{T} A(t)\mu(dt)=\infty$, then $\theta(A,r)=0$ for every $r\in \R$.

\end{enumerate}
\end{lem}

\fa{The second result proves a continuity in levels for processes with non-negative correlations.
\begin{thm}\label{lem:ohad}
The function $r\mapsto \theta(A,r)$ is continuous, i.e. the exponent is continuous in its levels.
\end{thm}
%
%
%
%
A previous result in this direction is that of \cite[Theorem 3.1]{Li_Shao}, who showed continuity of the exponent under the assumption that the correlation function $\rho(.)$ is strictly decreasing. More recently, \cite[Lemma 1.1]{FFM} proves a significantly improved version of continuity in levels, which allows for negative correlations, at the expense of mild integrability assumptions on the spectral measure. Theorem \ref{lem:ohad} shows that for non-negative correlation functions, no extra assumption on the spectral measure is necessary for continuity in levels.}
\\

We will now focus on comparing persistence of different processes. In this direction, we first state a lemma which allows us to compare persistence of two Gaussian vectors in $\R^n$.

\begin{lem}\label{lem:op}
Suppose $\{Z_n(i)\}_{1\le i\le n}$ and $\{Y_n(i)\}_{1\le i\le n}$ are two triangular arrays of centered Gaussian processes with positive definite covariance matrices $A_n,B_n$ respectively, such that 
\begin{align}\label{eq:compact}
\limsup_{n\to\infty}\Big(||A_n||_2+||A_n^{-1}||_2+||B_n||_2+||B_n^{-1}||_2\Big)<\infty,
\end{align}
where $||.||_2$ denotes the Euclidian operator norm/largest eigenvalue of a symmetric matrix. Assume further that $\lim_{n\to\infty}||A_n-B_n||_2=0$. Then we have
$$\limsup_{n\to\infty}\frac{1}{n}\left|\log \P(\max_{1\le i\le n}Y_n(i)<r)-\log \P(\max_{1\le i\le n}Z_n(i)<r)\right|=0.$$

\end{lem}

Using this lemma, we first prove a continuity lemma for persistence exponents for discrete time Gaussian processes. This significantly improves \cite[Lemma 5.1]{AMZ} by getting the same conclusion under much weaker hypotheses.


\begin{lem}\label{lem:DM2d}
For every positive integer $k$, let $\{Z_k(t)\}_{t\in \N}$ be a discrete time centered Gaussian process with non-negative correlation function $A_k(.,.)$. Further, let $\{Z_\infty(t)\}_{t\in \N}$ be a stationary centered Gaussian process with non-negative correlation function $A_\infty(.)$. Assume that $A_k(s,s+\tau)$ converges to $A_\infty(\tau)$ as $k\to\infty$, uniformly in $s\in \N_0:=\N\cup \{0\}$. Suppose further that with $g_k(\tau):=\sup_{s\in \N_0}A_k(s,s+\tau)$ we have
\begin{align}\label{eq:DM1d}
\lim_{L\to\infty}\limsup_{k\to\infty}\sum_{i=L}^\infty g_k(i)=0.
\end{align}

Then for every $r\in \R$ we have
\begin{align}\label{eq:conclusion_d}
\limsup_{k,n\to\infty}\frac{1}{n}\Big|\log \P(\sup_{1\le i\le n}Z_k(i)<r)-\log \P(\sup_{1\le i\le n}Z_\infty(i)<r)\Big|=0.
\end{align}

\end{lem}

Lifting the last result to the continuous setting, we prove the following lemma, which we will use in the sequel to prove all the main results of this paper.

\begin{thm}\label{lem:DM2}
For every positive integer $k$, let $\{Z_k(t)\}_{t\in \R_{\ge 0}}$ be a continuous time centered Gaussian process with non-negative correlation function $A_k(.,.)$ and continuous sample paths. Further, let $\{Z_\infty(t)\}_{t\in \R_{\ge 0}}$ be a stationary centered Gaussian process with continuous sample paths and non-negative correlation function $A_\infty(.)$. Assume that $A_k(s,s+\tau)$ converges to $A_\infty(\tau)$ as $k\to\infty$, uniformly in $s\geq 0$. Suppose further that the following conditions hold.

\begin{enumerate}
\item[(a)]
Setting $g_k(\tau):=\sup_{s\geq 0}A_k(s,s+\tau)$, for every positive integer $\ell$ we have
\begin{align}\label{eq:DM1}
\lim_{L\to\infty}\limsup_{k\to\infty}\sum_{i=L}^\infty g_k(i/\ell)=0.
\end{align}

\item[(b)]
There exists an $\eta>1$ such that 
\begin{align}
\label{eq:DM3}
\limsup_{\eps\to 0}|\log \eps|^\eta \sup_{1\le k\le \infty,s\ge 0, \tau\in [0,\eps]}(1-A_k(s,s+\tau))<\infty.
\end{align}

\item[(c)]

The limiting process $\{Z_\infty(t)\}_{t\in \R_{\ge 0}}$ has a persistence exponent which is sampling continuous, i.e.
\begin{align}\label{eq:cont_1}
\limsup_{\ell\to\infty}\limsup_{T\to\infty}\frac{1}{T}\Big|\log\P(\sup_{t\in [0,T]}Z_\infty(t)<r)-\log \P(\max_{1\le i\le \lceil T\ell \rceil}Z_\infty(i/\ell)<r)\Big|=0.
\end{align}
Then for every $r\in \R$ we have
\begin{align}\label{eq:conclusion}
\limsup_{k,T\to\infty}\frac{1}{T}\Big|\log \P(\sup_{t\in [0,T]}Z_k(t)<r)-\log \P(\sup_{t\in [0,T]}Z_\infty(t)<r)\Big|=0.
\end{align}
\end{enumerate}
\end{thm}

The above lemma is a significant generalization of previous versions of similar continuity results in \cite[Lemma 3.1]{DM2} and \cite[Theorem 1.6]{DM}. While the previous lemmas required a supremum decay control over $A_k(s,s+\tau)$, the current lemma replaces this by a summability condition, cf.\ \eqref{eq:DM1}. In particular, none of the previous results can be used to prove Theorem \ref{thm:summble} in this generality. Below we comment on sufficient conditions for verifying \eqref{eq:DM1} and \eqref{eq:cont_1}, respectively.

\begin{rem}\label{rem:DM1}
A sufficient condition for \eqref{eq:DM1} is that $\sup_{s\ge 0,k\in\N}A_k(s,s+\tau)\le g(\tau)$, where $g$ satisfies one of the following conditions:
\begin{enumerate}
\item[(i)]
$
\limsup\limits_{\tau\to\infty}\frac{\log g(\tau)}{\log \tau}<-1
$ (thus implying \cite[Theorem 1.6]{DM});

\item[(ii)]
$g$ is regularly varying and integrable
(thus implying \cite[Lemma 3.1]{DM2});

\item[(iii)]
$g$ is non-increasing and integrable.
\end{enumerate}
We do note that neither of these sufficient conditions, (i)--(iii), is enough to prove Theorem~\ref{thm:summble} in full generality, and we do need the full strength of Theorem~\ref{lem:DM2}. For all other results in this paper, the above sufficient conditions are enough.
\end{rem}

\begin{rem}\label{eq:cont_2020}
Condition \eqref{eq:cont_1} essentially demands that for the limiting Gaussian process, the persistence exponents obtained by discrete sampling in finer and finer grids converge to the persistence exponent of the continuous process. A sufficient condition for \eqref{eq:cont_1} is that  any of the conditions (i)--(iii) of Remark~\ref{rem:DM1} holds with $g(\tau)=A_\infty(\tau)$,
which can be verified by an application of \cite[Theorem 1.6]{DM} or \cite[Lemma 3.1]{DM2}. 
\end{rem}
\begin{rem}\label{rem:orn}
It is easy to verify that the scaled Ornstein-Uhlenbeck process with correlation function $A_\infty(\tau)=e^{-\alpha|\tau|}$ satisfies all the conditions (i)--(iii) of Remark~\ref{rem:DM1}, for any $\alpha>0$.  A fact that we will repeatedly use in this paper is that the scaled Ornstein-Uhlenbeck process has persistence exponent $\alpha$ (cf.\ e.g.\ the proof of  \cite[Lem 2.5]{DPSZ}).
\end{rem}

%
%
%
%
%
%
%

The following lemma is a modification of Theorem~\ref{lem:DM2} to the case where the limiting correlation is either non-integrable or degenerate, to deduce that the corresponding exponents converge to $0$ or $\infty$, respectively.

\begin{lem}\label{lem:DM22}
Suppose $\{A_k(\tau)\}_k$ is a sequence of stationary non-negative correlation functions on $\mathbb{T}$. Fix $r\in\R$.

\begin{enumerate}
\item[(a)]
Suppose \eqref{eq:DM3} holds, and $\lim\limits_{k\to\infty}A_k(\tau)=A_\infty(\tau)$ for every $\tau> 0$, and $\theta(A_\infty,r)=0$.
Then we have
$\lim_{k\to\infty}\theta(A_k{,r})=0$.

\item[(b)] 
Suppose \eqref{eq:DM1} holds, and  $\lim\limits_{k\to\infty}A_k(\tau)=0$ for every $\tau>0$. Then we have
$\lim\limits_{k\to\infty}\theta(A_k,r)=\infty.$
\end{enumerate}
\end{lem}

\section{Proof of Theorem \ref{thm:summble}} \label{sec:proofofsummable}


The following definition is crucial for the notation in the rest of this paper.
\begin{defn}
Extend $\sigma(.), s(.)$ to positive reals by setting $\sigma(x)=\sigma(\lceil x\rceil)$, and $\minorchanges{s^2(t)=\int_0^t\sigma^2(x)dx}$. Let $w(.)$ denote the inverse of $s(.)$, i.e. $w(s(u))=s(w(u))=u$ for all $u>0$.
\end{defn}

After interpolating the sequences $\rho(.)$ and $\sigma(.)$, we do the same with the covariances of $\{S_\ell\}$.

\begin{defn} \label{defn:F}
For any two positive reals $\ell_1,\ell_2$  let $$F_{\rho,\sigma}(\ell_1,\ell_2):=\int_0^{\ell_1}\int_0^{\ell_2}\sigma(x)\sigma(y)\rho(\lceil x\rceil-\lceil y\rceil)dxdy.$$
In particular, if $\ell_1,\ell_2$ are positive integers, then we have
$$F_{\rho,\sigma}(\ell_1,\ell_2)=\sum_{i=1}^{\ell_1}\sum_{j=1}^{\ell_2}\sigma(i)\sigma(j)\rho(i-j).$$
\end{defn}


We proceed with stating three lemmas which will be used to prove Theorem \ref{thm:summble}, the proofs of which we defer to the end of the section.

\begin{lem} \label{lem:sumlim1}
Let $\rho$ satisfy \eqref{eq:summable}, and assume that $\sigma$ satisfies \eqref{eq:sigma2} and \eqref{eq:sigma}. Then for every $b\ge 1$ 
\begin{align}\label{eq:limit_sum1}
\lim_{u\to\infty}\frac{F_{\rho,\sigma}(w(u), w(bu))}{u^2}=1+2\sum_{\ell=1}^\infty \rho(\ell).
\end{align}

\end{lem}

\begin{defn}
For two positive functions $f$ and $g$ depending on arguments $x_1$, $\ldots$, $x_n$, $z_1$, $\ldots$, $z_m$, the expression $f\lesssim_{z_1,\ldots,z_n} g$ means the existence of a finite positive constant $C=C(z_1,\ldots,z_n)$, such that $f\le C g$ for all $x_1,\ldots,x_n$ (possibly in a certain range). In particular, the notation $f\lesssim g$ implies the existence of a universal constant $C$ such that $f\le C g$ for all arguments (possibly in a certain range).
\end{defn}

\begin{lem}\label{lem:tight_1}
Let $\rho$ satisfy \eqref{eq:summable}. Then under no assumptions on $\sigma$ we have for any $b\geq 1$

\begin{align}\label{eq:tight_1}
\int_{w(u)}^{w(bu)}\int_{0}^{w(bu)}\sigma(x)\sigma(y)\rho(\lceil x \rceil-\lceil y\rceil)dydx\lesssim_\rho bu^2\sqrt{b^2-1}.
\end{align}

\end{lem}

\begin{lem}\label{lem:2020_1} Assume $\rho(.)$ satisfies (\ref{eq:summable}) and $\sigma(.)$ satifies (\ref{eq:sigma2}), (\ref{eq:sigma}), and (\ref{eq:log_concave}).
With $$\tilde{g}_u(\tau):=\frac{F_{\rho,\sigma}(w(u), w(e^\tau u))}{u^2 e^\tau}$$ we have
$$\lim_{L\to\infty} \limsup_{u\to\infty}\int_L^\infty \tilde{g}_u(\tau)d\tau=0.$$
\end{lem}


\begin{proof}[ of Theorem \ref{thm:summble}]
{\it Step 1: Reduction to convergence of the continuous-time interpolation.}

To begin, define a Gaussian process $(0,\infty)$ by setting
$X(u):=\int_0^u \sigma(\lceil v\rceil)\xi_{\lceil v\rceil}dv$, and note that $X(u)=S_u$ for all positive integers $u$.
Thus $X(.)$ is just the linear interpolation of the partial sums $\{S_\ell\}_{\ell\ge 1}$, and consequently 
$$\P(\max_{1\le \ell\le n}S_\ell<0)=\P(\sup_{u\in [0,n]}X(u)<0).$$ 
For any positive integer $k$ we define a Gaussian process on $[0,\infty)$ by setting  
$X_k(t):=k^{-1}X(w(ke^t))$, $t\ge 0$, and claim that
\begin{align}\label{eq:dm1}
\lim_{k,T\to\infty}\frac{1}{T}\log \P(\sup_{t\in [0,T]}X_k(t)<0)=-1.
\end{align}
Given \eqref{eq:dm1}, using Slepian's Lemma along with non-negativity of $\rho(.),\sigma(.)$, we get
\begin{align*}
\P(\sup_{u\in [0,n]}X(u)<0)&=\P(\sup_{u\in [0,w(k)]}X(u)<0,\sup_{t\in[0,\log(s(n)/k)]}X_k(t)<0)\\
&\ge \P(\sup_{u\in [0,w(k)]}X(u)<0)\cdot \P(\sup_{t\in [0,\log(s(n)/k)]}X_k(t)<0)\\
&=\P(\max_{1\le \ell \le w(k)}S_\ell<0)\cdot \P( \sup_{t\in [0,\log(s(n)/k)]}X_k(t)<0)\\
&\ge \Big(\frac{1}{2}\Big)^{w(k)}\cdot \P(\sup_{t\in [0,\log(s(n)/k)]}X_k(t)<0),
\end{align*}
which on taking limits as $n\to\infty$ gives 
\begin{align*}
&\liminf_{n\to\infty}\frac{1}{\log s(n)}\log \P(\sup_{t\in [0,n]}X(t)<0)\\
&\ge \liminf_{n\to\infty}\frac{1}{\log s(n)}\log \P(\sup_{t\in [0,\log(s(n)/k)]}X_k(t)<0)\\
&=\liminf_{n\to\infty}\frac{1}{ s'(n)}\log \P(\sup_{t\in [0,s'(n)]}X_k(t)<0),
\end{align*}
where $s'(n)=\log(s(n)/k)$ tends to $+\infty$ as $n\to\infty$ with $k$ fixed. On letting $k\to\infty$ on both sides of the above equation, invoking \eqref{eq:dm1} gives
$$\liminf_{n\to\infty}\frac{1}{\log s(n)}\log \P(\sup_{t\in [0,n]}X(t)<0)\ge -1,$$
thus giving the lower bound of the theorem.
The corresponding upper bound follows on noting that
$$\P(\sup_{t\in[0,n]}X(t)<0)\le \P(\sup_{t\in [0,\log(s(n)/k)]}X_k(t)<0),$$
and invoking \eqref{eq:dm1} again.
\\

{\it Step 2: Verification of \eqref{eq:dm1}.} To this effect, setting $u:=ke^s$, for $s\le t$ we have
\begin{align*}
\lim\limits_{k\to\infty}\cov(X_k(s),X_k(t))=&\lim\limits_{k\to\infty}\frac{F_{\rho,\sigma}( w(k e^s),w(k e^t) )}{k^2}\\
=&e^{2s}\lim\limits_{u\to\infty}\frac{F_{\rho,\sigma}(w(u) , w(u e^{t-s})) }{{u}^2}\\
=& e^{2s}\Big(1+2\sum_{\ell=1}^\infty \rho(\ell)\Big),
\end{align*}
where the last equality uses \eqref{eq:limit_sum1}. This gives
$$\lim_{k\to\infty}\corr(X_k(s),X_k(t))=\frac{e^{2s}}{e^{s} e^{t}}=e^{-|s-t|},$$
which is the correlation function of a scaled Ornstein Uhlenbeck process with persistence exponent $1$, by Remark \ref{rem:orn}. From this, the desired conclusion then follows on using Theorem~\ref{lem:DM2}, where we need to verify the conditions of the lemma. To this effect, first note that \eqref{eq:cont_1} holds for the Ornstein Uhlenbeck process, by Remark \ref{rem:orn}. Proceeding to verify \eqref{eq:DM3}, for $s\ge 0,\tau\in [0,1]$ setting $u:=ke^s\in [1,\infty)$ we have
\begin{align*}
1-\corr\Big(X_k(s),X_k({s+\tau})\Big)& =1-\frac{F_{\rho,\sigma}(w(u),w(ue^\tau))}{\sqrt{F_{\rho,\sigma}(w(u),w(u))}\sqrt{F_{\rho,\sigma}(w(ue^\tau), w(ue^\tau))}}\\
&\le 1-\frac{F_{\rho,\sigma}(w(u),w(ue^\tau))}{F_{\rho,\sigma}(w(ue^\tau), w(ue^\tau))}\\
&=\frac{F_{\rho,\sigma}(w(ue^\tau),w(ue^\tau))-F_{\rho,\sigma}(w(u),w(ue^\tau))}{F_{\rho,\sigma}(w(ue^\tau),w(ue^\tau))}\\
& \lesssim_{\rho,\sigma} e^{-\tau} \sqrt{e^{2\tau}-1},
\end{align*}
where the last inequality uses \eqref{eq:tight_1} along with \eqref{eq:limit_sum1}.
This verifies \eqref{eq:DM3}. It thus remains to verify \eqref{eq:DM1}, 
for which setting $u=ke^s$  we have
\begin{align*}
\corr(X_k(s),X_k(s+\tau))
&=\frac{F_{\rho,\sigma}(w( k e^s) , w(k e^{s+\tau} ))}{\sqrt{F_{\rho,\sigma}( w(k e^s) , w(k e^{s}) )}\sqrt{F_{\rho,\sigma}(w( k e^{s+\tau}) , w(k e^{s+\tau}) )}}\\
&=\frac{F_{\rho,\sigma}( w(u) , w(ue^\tau) )}{\sqrt{F_{\rho,\sigma}( w(u) , w(u) )}\sqrt{F_{\rho,\sigma}( w(ue^\tau) , w(u e^\tau))}}\\
&\lesssim_{\rho,\sigma}\frac{F_{\rho,\sigma}(w(u),w(e^\tau u))}{e^\tau u^2}=\tilde{g}_u(\tau),
\end{align*}
where $\tilde{g}_u(.)$ is as in Lemma \ref{lem:2020_1}, and the last inequality uses \eqref{eq:limit_sum1}. To verify \eqref{eq:DM1} it thus suffices to show that
\begin{align}\label{eq:DM1.5}
\lim_{L\to\infty}\limsup_{u\to\infty}\sum_{i=L}^\infty \tilde{g}_u(i/\ell)=0
\end{align}
for every positive integer $\ell$. To this effect, for $\tau\in [i/\ell, (i+1)/\ell]$ we have
\begin{align*}
\tilde{g}_u(i/\ell)&=\frac{F_{\rho,\sigma}( w(u) , w(ue^{i/\ell}) )}{e^{i/\ell} u^2}\\
&=e^{-i/\ell} u^{-2}\int_0^{w(u)}\int_0^{w(ue^{i/\ell})} \sigma(\lceil x\rceil)\sigma(\lceil y\rceil)\rho(\lceil x\rceil-\lceil y\rceil)dxdy\\
&\le e^{1/\ell-\tau} u^{-2}\int_0^{w(u)}\int_0^{w(ue^{\tau})} \sigma(\lceil x\rceil)\sigma(\lceil y\rceil)\rho(\lceil x\rceil-\lceil y\rceil)dxdy\\
&= e^{1/\ell} \tilde{g}_u(\tau).
\end{align*}
This immediately gives
\begin{align*}
\sum_{i=L}^\infty \tilde{g}_u(i/\ell)\le \ell e^{1/\ell}\sum_{i=L}^\infty \int_{i/\ell}^{(i+1)/\ell} \tilde{g}_u(\tau)d\tau= \ell e^{1/\ell}\int_{L/\ell}^\infty \tilde{g}_u(\tau)d\tau,
\end{align*}
from which \eqref{eq:DM1.5} follows on using 
Lemma \ref{lem:2020_1}. This completes the proof of the theorem.
\end{proof}

\begin{proof}[ of Lemma \ref{lem:sumlim1}] To begin note that
\begin{eqnarray}
\notag F_{\rho,\sigma}(w(u),w(bu))&\le &F_{\rho,\sigma}(\lceil w(u)\rceil, \lceil w(bu) \rceil)\\
&= &\notag\sum_{i=1}^{\lceil w(u)\rceil}\minorchanges{\sigma^2(i)}\ +2\sum_{\ell=1}^{\lceil w(u)\rceil -1}\rho(\ell)\sum_{i=1}^{\min(\lceil w(u)\rceil,\lceil w(bu)\rceil-\ell)}\sigma(i)\sigma(i+\ell)\\
\notag&&+\sum_{\ell=\lceil w(u)\rceil}^{\lceil w(bu)\rceil-1}\ \rho(\ell) \sum_{i=1}^{\min(\lceil w(u)\rceil,\lceil w(bu)\rceil-\ell)}\sigma(i)\sigma(i+\ell)\\
%
&=:&\sum_{\ell=0}^\infty \rho(\ell) \beta_{\ell,b}(u)\label{eq:sumlim2},
\end{eqnarray}
where 
$$
\beta_{\ell,b}(u):=\begin{cases} \sum_{i=1}^{\lceil w(u)\rceil} \sigma^2(i) & \text{ if }\ell=0
\\
2\sum_{i=1}^{\min(\lceil w(u)\rceil,\lceil w(bu)\rceil-\ell)}\sigma(i)\sigma(i+\ell)&\text{ if }1\le \ell\le \lceil w(u)\rceil-1\\
\sum_{i=1}^{\min(\lceil w(u)\rceil,\lceil w(bu)\rceil-\ell)}\sigma(i)\sigma(i+\ell)&\text{ if }\lceil w(u)\rceil \le \ell \le \lceil w(bu)\rceil -1\\
0&\text{ if }\ell \ge \lceil w(bu)\rceil.
\end{cases}
$$

Now for any $b>1$ we have $w(bu)-w(u)\rightarrow \infty$, as $w(bu)-w(u)\le K$ for some $K$ fixed along a subsequence in $u$ diverging to $+\infty$ implies
$b=\frac{bu}{u}=\frac{s(w(bu))}{s(w(u))}\le \frac{s(w(u)+K)}{s(w(u))}$, the right hand side of which converges to $1$ along the same subsequence, using \eqref{eq:sigma}. Thus for all $u$ large enough we have $w(bu)-\ell \ge w(u)$, and so
\begin{align*}
\frac{1}{u^2}\beta_{\ell,b}(u)=\frac{2}{u^2}\sum_{i=1}^{\lceil w(u)\rceil}\sigma(i)\sigma(i+\ell),
\end{align*}
which converges to $2$ as $u\to\infty$, 
invoking \eqref{eq:sigma} and \eqref{eq:sigma2}. Also, for any $u>0, \ell\ge 1$ we have, by the Cauchy-Schwarz inequality,
\begin{align*}
\notag \beta_{\ell,b}(u)\le 2\sum_{i=1}^{\min(\lceil w(u)\rceil,\lceil w(bu)\rceil-\ell)}\sigma(i)\sigma(i+\ell)
\le 2s(\lceil w(u)\rceil) s(\lceil w(bu)\rceil)\lesssim bu^2.
\end{align*}
which along with the Dominated Convergence theorem gives
$$\limsup_{u\to\infty}\frac{F(w(u),w(bu))}{u^2}\le 1+2\sum_{\ell=1}^\infty \rho(\ell),$$ thus giving the upper bound in \eqref{eq:limit_sum1}. The corresponding lower bound follows on noting that
\begin{eqnarray*}
F_{\rho,\sigma}(w(u),w(bu))
&\ge & F_{\rho,\sigma}(\lceil w(u)\rceil-1, \lceil w(u) \rceil-1)\\
&=&\sum_{i,j=1}^{\lceil w(u)\rceil-1}\sigma(i)\sigma(j)\rho(i-j)\\
&=&\sum_{i=1}^{\lceil w(u)\rceil-1}\sigma^2(i)+\sum_{\ell=1}^{\lceil w(u)\rceil -2}\rho(\ell)\sum_{i=1}^{\lceil w(u)\rceil -\ell-1}\sigma(i)\sigma(i+\ell),
\end{eqnarray*}
and using a similar argument as in the upper bound. 
\end{proof}


\begin{proof}[ of Lemma \ref{lem:tight_1}]
Using the Cauchy- Schwarz inequality, the left hand side of \eqref{eq:tight_1} can be bounded as follows:

\minorchanges{
\begin{align*}
&\int_{w(u)}^{w(bu)} \int_{0}^{w(bu)} \sigma(x)\sigma(y)\rho(\lceil x \rceil-\lceil y\rceil)dydx\\
&\le \sqrt{\int_{w(u)}^{w(bu)}\int_0^{w(bu)}\sigma^2(x)\rho(\lceil x \rceil-\lceil y\rceil)dydx} \sqrt{\int_{w(u)}^{w(bu)}\int_0^{w(bu)}\sigma^2(y)\rho(\lceil x \rceil-\lceil y\rceil)dydx}\\
&= \sqrt{\int_{w(u)}^{w(bu)}\sigma^2(x)\Big[ \int_0^{w(bu)}\rho(\lceil x \rceil-\lceil y\rceil)dy\Big]dx} \sqrt{\int_0^{w(bu)}\sigma^2(y)\Big[\int_{w(u)}^{w(bu)}\rho(\lceil x \rceil-\lceil y\rceil)dx\Big]dy}\\
& \lesssim_\rho \sqrt{\int_{w(u)}^{w(bu)}\sigma^2(x)dx} \sqrt{\int_0^{w(bu)}\sigma^2(y)dy},
\end{align*}}
where the last line uses the fact that $\rho$ is integrable. The last line equals $$\sqrt{s(w(bu))^2-s(w(u))^2} s(w(bu))=bu^2\sqrt{b^2-1},$$ as desired.
%
%
\end{proof}

\begin{proof}[ of Lemma \ref{lem:2020_1}]
{\it Step 1: We first treat the integral $\int_0^{w(u)} \int_0^{w(u)}$.} By the Cauchy-Schwarz inequality,
\begin{eqnarray}
&& \int_L^\infty \frac{e^{-\tau}}{u^2}\, \int_0^{w(u)} \int_0^{w(u)} \sigma(x)\sigma(y)\rho(\lceil y\rceil - \lceil x\rceil ) d y d x d \tau
 \notag \\
&\le&
 \int_L^\infty \frac{e^{-\tau}}{u^2}\, \sqrt{\int_0^{w(u)} \int_0^{w(u)} \minorchanges{\sigma^2(x)} \rho(\lceil y\rceil - \lceil x\rceil ) d y d x} 
\notag  \\
&& \cdot ~\sqrt{\int_0^{w(u)} \int_0^{w(u)} \minorchanges{\sigma^2(y)} \rho(\lceil y\rceil - \lceil x\rceil ) d y d x}  d \tau
\notag \\
&=&
\frac{e^{-L}}{u^2}\, \int_0^{w(u)} \minorchanges{\sigma^2(x)} \int_0^{w(u)}  \rho(\lceil y\rceil - \lceil x\rceil ) d y d x 
 \notag \\ 
&\lesssim_\rho&
\frac{e^{-L}}{u^2}\, \int_0^{w(u)} \minorchanges{\sigma^2(x)} d x = \frac{e^{-L}}{u^2}\, s(w(u))^2  = e^{-L}. \label{eqn:20-02-27-unimporantpartofint}
  \end{eqnarray}

{\it Step 2: Main part of the integral.} Fix $M\in\N$. Then we can write the remaining part in $\int_L^\infty \tilde{g}_u(\tau)d\tau$ as 
\begin{align}
\notag &\int_{L}^\infty \frac{e^{-\tau}}{u^2}\int_0^{w(u)}\int_{w(u)}^{w(e^\tau u)}\sigma(x)\sigma(y) \rho(\lceil y\rceil-\lceil x\rceil) {\bf 1}\{|x-y|> M\}dy dxd\tau
\\
\label{eq:b_2020}+&\int_{L}^\infty \frac{e^{-\tau}}{u^2}\int_0^{w(u)}\int_{w(u)}^{w(e^\tau u)}\sigma(x)\sigma(y) \rho(\lceil y\rceil-\lceil x\rceil) {\bf 1}\{|x-y|\le M\}dy dx d \tau.
\end{align}
Since $\rho$ satisfies \eqref{eq:summable}, we can write $\rho(i)=\tilde{\rho}(i)h(i)$, where both $\tilde{\rho}, h$ are non-negative functions satisfying $\sum_{i=1}^\infty  \tilde{\rho}(i)<\infty$ and $\lim_{i\to\infty}h(i)=0$. Using this, the first term in  \eqref{eq:b_2020} can be bounded as follows:
  \begin{align}
\notag &\int_{L}^\infty \frac{e^{-\tau}}{u^2}\int_0^{w(u)}\int_{w(u)}^{w(e^\tau u)}\sigma(x)\sigma(y) \rho(\lceil y\rceil-\lceil x\rceil) {\bf 1}\{|x-y|>M\}dy dx d \tau
\\
\notag&\le \sup_{i\ge M-2}h(i)\int_{0}^\infty \frac{e^{-\tau}}{u^2}\int_0^{w(u)}\int_{w(u)}^{w(e^\tau u)}\sigma(x)\sigma(y) \tilde{\rho}(\lceil y\rceil-\lceil x\rceil) dy dxd \tau\\
\notag&=\sup_{i\ge M-2}h(i)\int_0^{w(u)}\int_{w(u)}^{\infty}\sigma(x)\sigma(y) \tilde{\rho}(\lceil y\rceil-\lceil x\rceil) \int_{\log \frac{s(y)}{u}}^\infty \frac{e^{-\tau}}{u^2}d\tau dy dx
\\
\notag&=\sup_{i\ge M-2}h(i)\int_0^{w(u)}\int_{w(u)}^{\infty}\frac{\sigma(x)}{u}\frac{\sigma(y)}{s(y)} \tilde{\rho}(\lceil y\rceil-\lceil x\rceil)  dy dx
\\
\notag&\lesssim_\sigma\sup_{i\ge M-2}h(i)\int_0^{w(u)}\int_{w(u)}^{\infty}\frac{\minorchanges{\sigma^2(x)}}{u s(x)} \tilde{\rho}(\lceil y\rceil-\lceil x\rceil)  dy dx
\\
\label{eq:b_2020_1}&\lesssim_\rho \frac{1}{u}\sup_{i\ge M-2}h(i)\int_0^{w(u)} \frac{\minorchanges{\sigma^2(x)}}{s(x)}dx=2 \sup_{i\ge M-2}h(i),
\end{align}
where the inequalities in the last two lines use \eqref{eq:log_concave} (because $x\leq w(u)\leq y$) and summability of $\tilde{\rho}$, respectively. Proceeding to bound the second term in \eqref{eq:b_2020}, using \eqref{eq:sigma} we have
\begin{align}
\notag &\int_{L}^\infty \frac{e^{-\tau}}{u^2}\int_0^{w(u)}\int_0^{w(e^\tau u)}\sigma(x)\sigma(y) \rho(\lceil x\rceil-\lceil y\rceil) {\bf 1}\{|x-y|\le M\}dy dx d \tau\\
\notag&\lesssim_{\sigma,M} \int_{L}^\infty \frac{e^{-\tau}}{u^2}\int_0^{w(u)}\int_0^{w(e^\tau u)}\minorchanges{\sigma^2(x)} \rho(\lceil x\rceil-\lceil y\rceil) dy dx d \tau\\
\label{eq:b_2020_2}&\lesssim_{\rho} \int_{L}^\infty \frac{e^{-\tau}}{u^2}\int_0^{w(u)}\minorchanges{\sigma^2(x)}  dxd \tau=e^{-L}.
\end{align}
Combining \eqref{eq:b_2020_1} and \eqref{eq:b_2020_2} with \eqref{eqn:20-02-27-unimporantpartofint} and \eqref{eq:b_2020} we have
$$\int_L^\infty \tilde{g}_u(\tau)d\tau\le C(\sigma,\rho)\sup_{i\ge M-2}h(i)+C(\sigma,\rho,M) e^{-L},$$
which converges to $0$ on letting $L\to\infty$ followed by $M\to\infty$, on using the fact that $\lim_{i\to\infty}h(i)=0$.
\end{proof}


\begin{proof}[ of Proposition \ref{ppn:easy}]
Note that $q_n$ is a non-increasing sequence in $n$, and so it suffices to show that $\liminf_{n\to\infty}q_n>0$. To this effect, we first claim that $S_k\stackrel{a.s.}{\to}S_\infty:=\sum_{i=1}^\infty \sigma(i)\xi_i$, where the sum on the right hand side converges almost surely. It then follows that $S_\infty\sim \mathcal{N}(0,\sigma^2)$ for some $\sigma<\infty$, and so 
$$\P(\sup_{k\ge K}S_k<0)\ge \P(S_\infty<-1)-\P(\sup_{k\ge K}|S_k-S_\infty|>1).$$
On letting $K\to\infty$ we have
$$\lim_{K\to\infty}\P(\sup_{k\ge K}S_k<0)\ge \P(S_\infty<-1),$$
and so there exists $K\ge 1$ such that $\P(\sup_{k\ge K}S_k<0)\ge \P(S_\infty<-1)/2$. An application of Slepian's Lemma along with the fact that $\{S_\ell\}_{\ell\ge 1}$ has non-negative correlation gives for $n\geq K$
\begin{align*}
\P(\max_{1\le \ell \le n}S_\ell<0)\ge  \P(\max_{1\le \ell\le K-1}S_\ell<0) \P(\max_{K\le \ell\le n}S_\ell <0)
\ge \Big(\frac{1}{2}\Big)^{K} \P(S_\infty>-1),
\end{align*}
which is positive, and hence the proof is complete. It thus remains to verify the almost sure convergence of $\{S_k\}$. 
\fa{To this effect, define a Gaussian process $\{\tilde{S}_k\}_{k\ge 1}$ by setting $$\tilde{S}_k:=C\sum_{i=1}^k \sigma(i)\tilde{\xi}_i,\quad \{\tilde{\xi}_i\}_{i\ge 1}\stackrel{iid}{\sim}N(0,1),\quad C:=\sqrt{2\sum_{i=0}^\infty \rho(i)}.$$
Then, for any $m\ge n\ge 1$, Cauchy-Schwarz inequality gives
\begin{align*}
\E(S_m-S_n)^2=\sum_{i,j=n+1}^m\sigma(i)\sigma(j)\rho(i-j)
\le &\sqrt{\sum_{i,j=n+1}^m \sigma^2(i) \rho(i-j)} \sqrt{\sum_{i,j=n+1}^m \sigma^2(j)\rho(i-j)}\\
\le &C \sum_{i=n+1}^m \sigma^2(i)
=\E(\tilde{S}_m-\tilde{S}_{n})^2.
\end{align*}
An application of Markov's inequality, symmetry, and Sudakov-Fernique inequality (\cite[Thm 2.2.3]{AT}) then gives, for any $\delta>0$,
\begin{align}\label{eq:sudakov}
\notag\P(\max_{n\le k\le m}|S_k-S_m|\ge \delta)\le &\frac{1}{\delta}\E \max_{n\le k\le m}|S_k-S_m|\\
\le& \frac{2}{\delta} \E \max_{n\le k\le m}(S_k-S_m)\le \frac{2}{\delta}\E \max_{n\le k\le m}|\tilde{S}_k-\tilde{S}_m|.
\end{align}
Since $\{\tilde{S}_k\}_{k\ge 1}$ are sums of iid random variables, an application of Kolmogorov's Maximal inequality gives that for any $\lambda>0$ we have
\begin{align*}
\P(\max_{n\le k\le m}|\tilde{S}_k-\tilde{S}_m|\ge \lambda)\le \frac{C}{\lambda^2}\sum_{k=m+1}^n\sigma^2(k),
\end{align*}
which on integrating gives that for any $\varepsilon>0$,
\begin{align*}
\E \max_{n\le k\le m}|\tilde{S}_k-\tilde{S}_m|=\int_0^\infty \P(\max_{n\le k\le m}|\tilde{S}_k-\tilde{S}_m|\ge \lambda)d\lambda\le \varepsilon+\frac{C}{\varepsilon}\sum_{k=n+1}^m\sigma^2(k).
\end{align*}
Plugging the choice $\varepsilon=\sqrt{C\sum_{k=n+1}^m\sigma^2(k)}$ gives the bound
\begin{align}\label{eq:sudakov2}
\E \max_{n\le k\le m}|\tilde{S}_k-\tilde{S}_m|\le 2\sqrt{C\sum_{k=n+1}^m\sigma^2(k)}.
\end{align}
Combining \eqref{eq:sudakov} and \eqref{eq:sudakov2} gives
$$\P(\max_{n\le k,\ell\le m}|S_k-S_\ell|\ge  2\delta)\le \P(\max_{n\le k\le m}|S_k-S_m|\ge \delta)\le \frac{4}{\delta}\sqrt{C\sum_{k=n+1}^m\sigma^2(k)},$$
which on letting $m\to\infty$ along with continuity in probability gives
$$\P(\max_{k,\ell\ge n}|S_k-S_\ell |\ge 2\delta)\le \frac{4}{\delta} \sqrt{C\sum_{k=n+1}^\infty\sigma^2(k)}.$$
Since $\lim\limits_{n\to\infty}\max_{k,\ell\ge n}|S_k-S_\ell |\le \max_{k,\ell\ge n}|S_k-S_\ell|$, letting $n\to\infty$ this gives
$$\P(\lim_{n\to\infty}\max_{k,\ell\ge n}|S_k-S_\ell|\ge 2\delta)=0.$$
Since $\delta>0$ is arbitrary, we have $\lim_{n\to\infty}\max_{k,\ell\ge n}|S_k-S_\ell|\stackrel{a.s.}{=}0$, i.e.~the sequence $\{S_k\}_{k\ge 1}$ is Cauchy almost surely. This proves almost sure convergence, and hence completes the proof of the proposition.}



\end{proof}

\begin{proof}[ of Proposition \ref{ppn:exp}]
Fixing integers $\tau\ge 0,\ell\ge 1,k\ge 1$ we have
\begin{align*}
\cov(S_{\ell+k},S_{\ell+\tau+k})=&\sum_{i=1}^{\ell+k} \sum_{j=1}^{\ell+\tau+k}e^{(i+j)\alpha}\rho(i-j)\\
=&e^{(2\ell+2k+\tau)\alpha}\sum_{i=0}^{\ell+k-1} \sum_{j=0}^{\ell+\tau+k-1}e^{-(i+j)\alpha}\rho(j-i-\tau),
\end{align*}
which gives
\begin{align}\label{eq:DZ}
\lim_{k\to \infty}\sup_{\ell\ge 1,\tau\ge 0}\Big|\frac{\corr(S_{\ell+k},S_{\ell+\tau+k})}{D_\alpha(\tau)}-1\Big|=0,
\end{align}
and so $D_\alpha(.)$ is a valid correlation function. To get the desired conclusion, we apply Lemma \ref{lem:DM2d}. It thus suffices to verify \eqref{eq:DM1d}, which follows on noting that \eqref{eq:DZ} implies $\corr(S_{\ell+k},S_{\ell+\tau+k})\lesssim_{\alpha,\rho} D_\alpha(\tau)$, which is summable using \eqref{eq:summable}.
%
%
\end{proof}

\section{The non-summable case} \label{sec:proofnonsummable}

\subsection{Proof of Theorem \ref{thm:1}}

We begin by stating two lemmas which will be used to prove Theorem \ref{thm:1}. We defer the proofs of these two results to the appendix.

\begin{lem}\label{lem:f}  Let $H\in (1/2,1)$, $p+H>0$, and for $x\ge 1, \alpha\in \R$ let 
$\psi_\alpha(x):=\int_1^x y^{\alpha}dy$.
\begin{enumerate}
\item[(a)]
For any $b\ge N\ge 1$ we have
\begin{align}
\label{eq:f1}
f_{p,H}(1,b)\le &f_{p,H}(1,N)+\Big(1-\frac{1}{N}\Big)^{2H-2}  \frac{\psi_{p+2H-2}(b)}{p+1},\\
\label{eq:f2}
f_{p,H}(1,b)\ge &f_{p,H}(1,1)+\frac{\psi_{p+2H-2}(b)}{p+1}.
\end{align}

\item[(b)]
$$f_{p,H}(1,1)=\frac{\Gamma(p+1) \Gamma(2H-1)}{(p+H)\Gamma(p+2H)}=\frac{\Beta(p+1,2H-1)}{p+H}.$$
\end{enumerate}
\end{lem}

\begin{lem}\label{lem:limit} Let $H\in (1/2,1)$, $p+H>0$. Then we have

\begin{align}\label{eq:limit}
\lim_{u\to\infty}\sup_{b\ge 1}\Big|\frac{F_{\rho,\sigma}(u, bu)}{u^{2p+2H}  f_{p,H}(1,b)}-\kappa\Big|=
0 ,
\end{align}
where $F_{\rho,\sigma}(.,.)$ is as in definition \ref{defn:F}, with $\rho$ satisfying \eqref{eq:rho} and $\sigma$ satisfying \eqref{eq:sigma_p}.

\end{lem}

%
%

Equipped with these lemmas, we can now prove Theorem \ref{thm:1}. 
\\

\begin{proof} [ of Theorem \ref{thm:1}]
As in the proof of Theorem \ref{thm:summble}, define the continuous time Gaussian process $X(.)$ on $(0,\infty)$ by setting
 $X(u):=\int_0^u \sigma(\lceil v\rceil)\xi_{\lceil v\rceil}dv$. 
For any positive integer $k$  define a Gaussian process on $[0,\infty)$ by setting  
$X_k(t):=k^{-(p+1/2)}X(ke^t)$. As in Step 1 of the proof of Theorem~\ref{thm:summble} (cf.\ (\ref{eq:dm1})), it suffices to show that
\begin{align}\label{eq:dm11}
\lim_{k,T\to\infty}\frac{1}{T}\log \P(\sup_{t\in [0,T]}X_k(t)<0)=-\theta(C_{p,H}).
\end{align}
%
For showing \eqref{eq:dm11}, note that for any $0\le s\le t$, setting $u:=ke^s$ we have
\begin{align*}
\lim_{k\to\infty}\cov(X_k(t),X_k(s))&=\lim_{u\to\infty}\frac{F_{\rho,\sigma}( k e^s,  k e^s e^{t-s})}{k^{2p+2H}}\\
&=e^{(2p+2H)s}\lim_{u\to\infty}\frac{F_{\rho,\sigma}( u,  ue^{t-s})}{u^{2p+2H}}\\
&=\kappa e^{(2p+2H)s}f_{p,H}(1,e^{t-s}),
\end{align*}
using \eqref{eq:limit}. This readily gives
$$\lim_{k\to\infty}\corr(X_k(t),X_k(s))=e^{(s-t)(p+H)}\frac{f_{p,H}(1,e^{t-s}) }{f_{p,H}(1,1)} =C_{p,H}(t-s), $$
thus verifying that $C_{p,H}(\tau)$ is a valid correlation function. For showing that the limit $\theta(p,H)\in (0,\infty)$ we invoke Lemma \ref{lem:exist}, so that it suffices to show that $\int_0^\infty C_{p,H}(\tau)d\tau<\infty$, which follows on using \eqref{eq:f1} to note that $f_{p,H}(1,e^\tau)\lesssim_{p,H}\max(\tau, e^{\tau(p+2H-1)})$, and so we have
\begin{align}\label{eq:int_new}
C_{p,H}(\tau)\lesssim_{p,H}\max(\tau e^{-\tau(p+H)}, e^{-\tau(1-H)}).
\end{align} 
To conclude \eqref{eq:dm11}, it thus remains to verify the conditions of Theorem~\ref{lem:DM2}. To this effect, note that \eqref{eq:cont_1} follows from \eqref{eq:int_new} and Remark \ref{eq:cont_2020}. 
It thus remains to verify \eqref{eq:DM1} and \eqref{eq:DM3}, which is done below.
\\

%

{\bf Verification of \eqref{eq:DM1}.} Use Lemma \ref{lem:limit} to note the existence of $M<\infty$ such that for all $u>M$ and $b\ge e$ we have $$ \frac{1}{2} u^{2p+2H} f_{p,H}(1,b)\le \frac{F_{\rho,\sigma}( u, bu )}{\kappa}\le  2 u^{2p+2H} f_{p,H}(1,b).$$
Thus for any $s>0,\tau>1$, noting that $u=ke^s >M$ gives
\begin{align*}
\corr(X_k(s),X_k(s+\tau))&=\frac{F_{\rho,\sigma}( u,  ue^{\tau})}{\sqrt{F_{\rho,\sigma}( u,u  )}\sqrt{F_{\rho,\sigma}( u e^{\tau}, u e^{\tau})}}\\
&\le 4C_{p,H}(\tau)\lesssim_{p,H}\max(\tau e^{-\tau(p+H)},e^{-\tau(1-H)}),
\end{align*}
where the last inequality uses \eqref{eq:int_new}.
This verifies \eqref{eq:DM1} via Remark \ref{rem:DM1}.
\\

{\bf Verification of \eqref{eq:DM3}.}
For verifying \eqref{eq:DM3}, for $s\ge 0,\tau\in [0,1]$ setting $u:=ke^s\in [1,\infty)$ we have
\begin{eqnarray}\label{eq:tight1.1}
\notag
&&F_{\rho,\sigma}(ue^\tau,ue^\tau)-F_{\rho,\sigma}(u,ue^\tau)
\\
\notag&=&\int_0^{1}\int_{u}^{ue^\tau}\sigma(\lceil x\rceil)\sigma(\lceil y\rceil)\rho(\lceil x\rceil-\lceil y\rceil)dxdy
\\
\notag
&&+\int_1^{u/2}\int_{u}^{ue^\tau}\sigma(\lceil x\rceil)\sigma(\lceil y\rceil)\rho(\lceil x\rceil-\lceil y\rceil)dxdy\\
&&
+\int_{u/2}^{ue^\tau}\int_{u}^{ue^\tau}\sigma(\lceil x\rceil)\sigma(\lceil y\rceil)\rho(\lceil x\rceil-\lceil y\rceil)dxdy.
\end{eqnarray}
Using  \eqref{eq:rho} and \eqref{eq:sigma_p} we get $\rho(\lceil x\rceil-\lceil y\rceil)\lesssim_\rho |x-y|^{2H-2}$, and $\sigma(\lceil y\rceil)\lesssim_\sigma y^p$ for $y\ge 1$. Consequently, the first term on the right hand side of \eqref{eq:tight1.1} can be bounded as follows:
\begin{align}\label{eq:tight2.1}
\notag&  \int_{0}^{1}\int_{u}^{ue^\tau} \sigma(\lceil x\rceil)\sigma(\lceil y\rceil) \rho(\lceil x\rceil-\lceil y\rceil)dxdy\\
&\lesssim_{\sigma,\rho} u^{2H-2} \int_{0}^{1}\int_{u}^{ue^\tau} x^p dxdy\lesssim_p u^{p+2H-1}\tau
\end{align}
Similarly, the second term on the right hand side of \eqref{eq:tight1.1} can be bounded as follows:
\begin{align}\label{eq:tight3.1}
  \notag&\int_{1}^{u/2}\int_{u}^{ue^\tau} \sigma(\lceil x\rceil)\sigma(\lceil y\rceil) \rho(\lceil x\rceil-\lceil y\rceil)dxdy\\
&\lesssim_{\sigma,\rho} u^{p+2H-2} \int_{1}^{u/2}\int_{u}^{ue^\tau}x^pdxdy\lesssim_\sigma u^{2p+2H}\tau.
\end{align}
Finally, the third term on the right hand side of \eqref{eq:tight1.1} can be estimated as
\begin{align}\label{eq:tight4.1}
\notag&\int_{u/2}^{ue^\tau}\int_{u}^{ue^\tau} \sigma(\lceil x\rceil)\sigma(\lceil x\rceil) \rho(\lceil x\rceil-\lceil y\rceil)dx dy\\
\notag&\lesssim_{\rho,\sigma} u^{2p+2H}\int_{1/2}^{ e^\tau} \int_{1}^{e^\tau} |x-y|^{2H-2}dxdy\\
\notag&=u^{2p+2H}\int_1^{e^\tau}\int_{1/2-x}^{e^\tau-x} |z|^{2H-2} dzdx\\
&\le u^{2p+2H}\int_1^{e^\tau}\int_{1/2-e}^{e-1} |z|^{2H-2}dz dx\lesssim_H u^{2p+2H}(e^\tau-1).
%
\end{align}

Combining \eqref{eq:tight2.1}, \eqref{eq:tight3.1} and \eqref{eq:tight4.1} along with \eqref{eq:tight1.1} gives
\begin{align}\label{eq:tight5.1}
 F_{\rho,\sigma}(ue^\tau,ue^\tau)-F_{\rho,\sigma}(u,ue^\tau)\lesssim_{\rho,\sigma}&{u}^{2p+2H}\tau.
\end{align}
Using \eqref{eq:tight5.1} and \eqref{eq:limit} gives, for all $\tau\in [0,1]$,
$$
\sup_{s\ge 0, k\ge 1}\left(1-\corr\Big(X_k(s),X_k(s+\tau)\Big)\right)\lesssim_{\rho,\sigma} \tau,
$$
which verifies \eqref{eq:DM3}, and hence completes the proof of the theorem.
\end{proof}

\subsection{Proof of Theorem \ref{exp:conti}}

\begin{proof}[ of Theorem \ref{exp:conti}] \textit{Step 1: Proof of continuity.}
Let $(p_k,H_k)$ be a sequence converging to $(p_\infty,H_\infty)$. We need to show that $$\lim_{k\to\infty}\theta({p_k,H_k})=\theta({p_\infty,H_\infty}).$$ Since  $\lim_{k\to\infty}C_{p_k,H_k}(\tau)=C_{p_\infty,H_\infty}(\tau)$, this will follow by another application of Theorem~\ref{lem:DM2}, once we verify the conditions of that lemma. Using \eqref{eq:f1} gives the existence of a continuous function $M(p,H)$ such that 
$$f_{p,H}(1,e^\tau)\le M(p,H) \max\Big(\tau,e^{\tau(p+2H-1)}\Big),$$
which gives
\begin{align*}
C_{p_k,H_k}(\tau) \le  M \max\Big(\tau e^{-(p_k+H_k)}, e^{-\tau(1-H_k)}\big),\quad M:=\sup_{k\ge 1} M(p_k,H_k) ,
\end{align*} and consequently
$$\sup_{k\ge 1,\tau\ge 0}\frac{\log C_{p_k,H_k}(\tau)}{\log \tau}=-\infty.$$
This verifies  \eqref{eq:DM1} via Remark \ref{rem:DM1}. 
The above display along with Remark \ref{eq:cont_2020} also verifies \eqref{eq:cont_1}. It thus suffices to verify \eqref{eq:DM3}. But this follows on noting that $C_{p_k,H_k}(\tau)\ge e^{-(p_k+H_k)\tau}$.
\\

\textit{Step 2:  Proof of \eqref{eq:limH}.} To begin note that
\begin{align}
\notag-\frac{\theta({p,H})}{1-H}=&\lim_{T\to\infty} \frac{1}{T}\log \P\Big(\sup_{t\in [0,\frac{T}{1-H}]}Z(t)<0\Big)\\
=&\lim_{T\to\infty}\frac{1}{T}\log \P\Big(\sup_{t\in [0,T]}Z\Big(\frac{t}{1-H}\Big)<0\Big)\label{eq:h=1}.
\end{align}
Let 
\begin{align*}
A_{p,H}(\tau):=&C_{p,H}\Big(\frac{\tau}{1-H}\Big)=e^{-\tau\frac{p+H}{1-H}} \frac{f_{p,H}(1,e^{\frac{\tau}{1-H}})}{f_{p,H}(1,1)}
\end{align*}
 denote the correlation of the process $\Big\{Z\Big(\frac{t}{1-H}\Big),t\ge 0\Big\}$. Using \eqref{eq:f1}, 
 on letting $H\uparrow 1$ followed by $N\to\infty$, for all $H$ such that $p+2H>1$ we get
\begin{align*}
\limsup_{H\uparrow 1} \frac{f_{p,H}(1, e^{\frac{\tau}{1-H}})}{e^{\tau\frac{p+2H-1}{1-H}}}\le \frac{1}{(p+1)^2}.
\end{align*}
A similar calculation using \eqref{eq:f2} gives the lower bound
and so we get $$\lim_{H\uparrow 1} \frac{f_{p,H}(1,e^{\frac{\tau}{1-H}})}{e^{\tau\frac{p+2H-1}{1-H}}}=\frac{1}{(p+1)^2},$$
which immediately gives that $$\lim_{H\uparrow 1}A_{p,H}(\tau)=e^{-\tau}.$$ This is the correlation function of the scaled Ornstein-Uhlenbeck process, which by Remark \ref{rem:orn} satisfies \eqref{eq:cont_1}, and has persistence exponent $1$. The desired conclusion will then follow from \eqref{eq:h=1} by invoking Theorem~\ref{lem:DM2}, once we verify the other two conditions of that lemma, namely \eqref{eq:DM1} and \eqref{eq:DM3}. 
\\

To this effect, again use \eqref{eq:f1} with $N=2$ to note that for all $H$ such that $p+2H>1$ we have
$$f_{p,H}(1,e^{\frac{\tau}{1-H}})\lesssim_{p} e^{\tau\frac{p+2H-1}{1-H}},$$
which gives
$A_{p,H}(\tau)\lesssim_p e^{-\tau},$ thus verifying \eqref{eq:DM1} via Remark \ref{rem:DM1}. 
Proceeding to verify \eqref{eq:DM3},  setting $\beta:=\beta_H:=\frac{1}{1-H}$ and using \eqref{eq:f2} along with the fact that $p+2H>1$ for all $H\sim 1$ we get
\begin{align*}
& f_{p,H}(1,1) - e^{-\beta \tau (p+H)} f_{p,H}(1,e^{\beta \tau})
\\
&\leq  f_{p,H}(1,1) - e^{-\beta \tau (p+H)} [ f_{p,H}(1,1) + \frac{(e^{\beta \tau})^{p+2H-1} -1}{(p+1)(p+2H-1)} ]
\\
&=  f_{p,H}(1,1)(1-e^{-\tau}) + \left( e^{-\beta \tau (p+H)} -e^{-\tau}\right) [\frac{1}{(p+1)(p+2H-1)}- f_{p,H}(1,1)]
\\
&\le f_{p,H}(1,1) \tau  + \left| e^{-\beta \tau (p+H)} -e^{-\tau}\right| \left|\frac{1}{(p+1)(p+2H-1)}- f_{p,H}(1,1)\right|
\\
&\le f_{p,H}(1,1) \tau  + ( \beta (p+H) + 1)\tau \left|\frac{1}{(p+1)(p+2H-1)}- f_{p,H}(1,1)\right|,
\end{align*}
where the last step follow from the fact that $|e^{-x}-e^{-y}| \leq x +y$. This gives
\begin{align*}
\frac{1-A_{p,H}(\tau)}{\tau}& =\frac{ f_{p,H}(1,1) - e^{-\beta \tau(p+H)} f_{p,H}(1,e^{\beta \tau})}{\tau f_{p,H}(1,1)}\\
& \le 1+(\beta(p+H)+1) \left|\frac{1}{(p+1)(p+2H-1)f_{p,H}(1,1)}-1\right|,
\end{align*}
and so to verify \eqref{eq:DM3} it suffices to show that the right hand side above stays bounded as $H\uparrow 1$, or equivalently, $(1-H)^{-1} |f_{p,H}(1,1)-\frac{1}{(p+1)^2}|$ stays bounded as $H\uparrow 1$. Use part (b) of Lemma \ref{lem:f} to note  that $f_{p,1}(1,1)=\frac{1}{(p+1)^2}$, and so it suffices to show that $\frac{\partial f_{p,H}(1,1)}{\partial H}$ is bounded in a neighborhood of $1$. 
But this follows on noting that the derivative is continuous in $H$, and converges as $H\uparrow 1$ to
$$2\int_0^1\int_0^1 x^p y^p \log |x-y|dxdy,$$ which is finite for $p>-1$.
\\



\textit{Step 3:  Proof of \eqref{eq:limH2}.} To begin, fixing $\delta>0$ and using \eqref{eq:f1} with $N=1+\delta$ gives
\begin{align*}
f_{p,H}(1,e^\tau)- f_{p,H}(1,1+\delta)\lesssim_p&\Big(\frac{\delta}{1+\delta}\Big)^{2H-2}\max\Big(\tau, e^{\tau(p+2H-1)}\Big),
\end{align*}
which gives
\begin{align*}
f_{p,H}(1,e^\tau)-(1+\delta)^{2p+2H}f_{p,H}(1,1)\lesssim_p \Big(\frac{\delta}{1+\delta}\Big)^{2H-2}\max\Big(\tau,e^{\tau(p+2H-1)}\Big),
\end{align*}
which gives 
\begin{align}\label{eq:orn_upper}
 C_{p,H}(\tau)-(1+\delta)^{2p+2H} e^{-\tau(p+H)}
\lesssim_p&\Big(\frac{\delta}{1+\delta}\Big)^{2H-2}\frac{\max\Big(\tau e^{-\tau(p+H)},e^{-\tau(1-H)}\Big)}{f_{p,H}(1,1)}.
\end{align}
On letting $H\downarrow 1/2$ followed by $\delta\to 0$ and noting that $f_{p,H}(1,1)\to \infty$ gives
\begin{align*}
\limsup_{H\downarrow 1/2}C_{p,H}(\tau)\le e^{-\tau(p+\frac{1}{2})}.
\end{align*}
The corresponding lower bound follows from the trivial bound
$C_{p,H}(\tau)\ge e^{-(p+H)\tau}$, giving $\lim_{H\downarrow 1/2}C_{p,H}(\tau)=e^{-(p+1/2)\tau}$, which is the correlation function of the scaled Ornstein-Uhlenbeck process which has persistence exponent $p+1/2$, by Remark \ref{rem:orn}, and satisfies \eqref{eq:cont_1}. The desired conclusion will then follow from Theorem~\ref{lem:DM2}, once we verify the conditions \eqref{eq:DM1} and \eqref{eq:DM3} of the lemma.
To this effect, \eqref{eq:DM1} follows from \eqref{eq:orn_upper} and Remark \ref{rem:DM1}, and 
\eqref{eq:DM3} follows on noting that $C_{p,H}(\tau)\ge e^{-(p+H)\tau}$, and so the proof is complete.
\\


\textit{Step 4:  Proof of \eqref{eq:limp}.} For any $\tau>0$  using \eqref{eq:f1} with $N=e^{\frac{\tau}{4}}$ gives
\begin{align*}
f_{p,H}(1,e^\tau)
\le e^{\frac{(p+H)\tau}{2}}f_{p,H}(1,1)+(1-e^{-\frac{\tau}{4}})^{2H-2} \frac{e^{\tau(p+2H-1)}}{{(p+1)(p+2H-1)}},
\end{align*}
which readily gives
\begin{align}\label{eq:ready}
 C_{p,H}(\tau)&\le  e^{-\frac{(p+H)\tau}{2}}+(1-e^{-\frac{\tau}{4}})^{2H-2}\frac{e^{-\tau(1-H)}}{f_{p,H}(1,1)(p+1)(p+2H-1)}\\
&\lesssim_{H,\tau} e^{-\frac{(p+H)\tau}{2}}+p^{2H-2},  \notag 
\end{align}
where the last inequality uses part (b) of Lemma \ref{lem:f}.
Thus letting $p\to\infty $ we get $\lim_{p\to\infty}C_{p,H}(\tau)=0$. Further,  using \eqref{eq:ready} for all $p\ge 1,\tau\ge 1$ we have 
$ C_{p,H}(\tau)\lesssim_H e^{-\frac{\tau H}{2}}+ e^{-\tau(1-H)}$, and so $C_{p,H}(\tau)$ satisfies \eqref{eq:DM1}. Thus conclusion then follows from part (b) of Lemma \ref{lem:DM22}.
\\


\textit{Step 5:  Proof of \eqref{eq:limpinfty}.}
Note that the stationary Gaussian process with correlation function $C_{p,H}(\tau/p)$ has persistence exponent $\theta(C_{p,H})/p$. We shall show that this sequence of correlation functions, when $p\to\infty$, converges to a non-integrable correlation function, and then invoke part (a) of Lemma \ref{lem:DM22}. Recall that

$$ C_{p,H}\left(\frac{\tau}{p}\right)
= e^{-(p+H) \tau / p} \, \frac{\int_0^1 \int_0^{e^{\tau/p}} x^p y^p |x-y|^{2H-2} d x d y}{\int_0^1 \int_0^1 x^p y^p |x-y|^{2H-2} d x d y}.
$$

Clearly, the first term tends to $e^{-\tau}$. In the integral, we set $x=1-u/p$ and $y=1-v/p$ and obtain
$$
 \int_0^1 \int_0^{e^{\tau/p}} x^p y^p |x-y|^{2H-2} d x d y = p^{-2H} \int_0^p \int_{(1-e^{\tau/p})p}^p (1-u/p)^p (1-v/p)^p |u-v|^{2H-2} d u d v.
$$
Therefore, as $p\to\infty$
$$
C_{p,H}\left(\frac{\tau}{p}\right) \to e^{-\tau} \, \frac{\int_0^\infty \int_{-\tau}^\infty e^{-u} e^{-v} |u-v|^{2H-2} d u d v}{\int_0^\infty \int_0^\infty e^{-u} e^{-v} |u-v|^{2H-2} d u d v} =: C_{\infty,H}(\tau).
$$


We now claim that function $C_{\infty,H}(.)$ is non-integrable. Indeed, denoting by $K^{-1}$ the constant in the denominator of $C_{\infty,H}$, we have, as $\tau\to\infty$,
\begin{eqnarray*}
&& \tau^{2-2H}C_{\infty,H}(\tau)
\\
\\
&=&K \, \int_0^\infty \int_0^{\infty} e^{-u} e^{-v} \big|\frac{u-v}{\tau}-1\Big|^{2H-2} d u d v
\\
\\
&\to &K\, \int_0^\infty \int_0^\infty e^{-u} e^{-v} d u d v,
\end{eqnarray*}
and so $C_{\infty,H}(\tau)$ is regularly varying and non-integrable.  This observation along with part (b) of Lemma \ref{lem:exist} shows that $\theta(C_{\infty,H})=0$. From this, the conclusion follows from part (a) of Lemma \ref{lem:DM22}, once we verify \eqref{eq:DM3}. But this follows on noting that $C_{p,H}(\tau/p)\ge e^{-\frac{p+H}{p}\tau}$.
\\

\textit{Step 6: Proof of \eqref{eq:limp+h}.}
We first claim that
\begin{align}
\limsup_{p\downarrow -H}\int_0^1 \int_1^\infty x^p y^p |x-y|^{2H-2} d x d y  < \infty.
 \label{lem:integralinp+Htozero}
 \end{align}
We first complete the proof of \eqref{eq:limp+h}, deferring the proof of \eqref{lem:integralinp+Htozero}. To this effect, note that the stationary Gaussian process with correlation function $C_{p,H}\left(\frac{\tau}{p+H}\right)$ has persistence exponent $\frac{\theta(C_{p,H})}{p+H}$. We shall show that this sequence of correlation functions, when $p\downarrow -H$, converges to $e^{-\tau}$, which is the correlation function of an Ornstein-Uhlenbeck process with $\alpha=1$, satisfying \eqref{eq:cont_1}, by Remark \ref{rem:orn}).

For this purpose, first note that
$$
\lim_{p\downarrow -H}f_{p,H}(1,1)=\frac{\Gamma(p+1)\Gamma(2H-1)}{(p+H)\Gamma(p+2H)}\to \infty.
$$
Second, note that
$$
f_{p,H}(1,e^{\frac{\tau}{p+H}}) = f_{p,H}(1,1) + \int_0^1 \int_1^{e^{\frac{\tau}{p+H}}} x^p y^p |x-y|^{2H-2} d x d y.
$$
By \eqref{lem:integralinp+Htozero}, the second term remains bounded when $p\downarrow -H$. Therefore,
$$
C_{p,H}\left(\frac{\tau}{p+H}\right)=e^{-\tau}\frac{f_{p,H}(1,e^{\frac{\tau}{p+H}})}{f_{p,H}(1,1)} \to e^{-\tau}.
$$
The conclusion follows by invoking Theorem~\ref{lem:DM2}, once we verify \eqref{eq:DM1} and \eqref{eq:DM3}.
Invoking \eqref{lem:integralinp+Htozero} we have $C_{p,H}(\tau/(p+H))\lesssim_H e^{-\tau}$ so that we get \eqref{eq:DM1} via Remark \ref{rem:DM1}. Further, \eqref{eq:DM3} is obtained from the observation that $C_{p,H}(\frac{\tau}{p+H})\ge e^{-\tau}$, and so \eqref{eq:limp+h} follows.
\\

It thus remains to verify \eqref{lem:integralinp+Htozero}. To this effect we have
\begin{eqnarray*}
&& \int_0^1 \int_1^\infty x^p y^p |x-y|^{2H-2} d x d y 
\\
&=& \int_0^1 \int_1^2 x^p y^p (x-y)^{2H-2} d x d y + \int_0^1 \int_2^\infty x^p y^p (x-y)^{2H-2} d x d y 
\\
&\leq & \int_0^1 \int_1^2  y^p (x-1)^{2H-2} d x d y + \int_0^1 \int_2^\infty x^p y^p (x-1)^{2H-2} d x d y 
\\
&\leq &  \frac{1}{1+p} \Big[ \int_1^2 (x-1)^{2H-2} d x  + \int_2^\infty  x^p(x/2)^{2H-2} d x \Big]
\\
&= &  \frac{1}{1+p} \Big[ \frac{1}{2H-1}  + 2^{2-2H} \int_2^\infty x^{p+2H-2}   d x \Big]\\
&\to &  \frac{1}{1-H} \Big[ \frac{1}{2H-1}  + 2^{2-2H} \int_2^\infty x^{H-2}   d x \Big],
\end{eqnarray*}
which is clearly finite as $H-2<-1$. This verifies \eqref{lem:integralinp+Htozero}, and hence completes the proof of \eqref{eq:limp+h}.
\end{proof}

 \section{Proof of the general tools} \label{sec:proofofgeneraltools}
 
 Throughout this section we carry out the proofs for $\mathbb{T}=\R_{\ge 0}$, noting that the proof for $\N$ is simpler, and follows by minor modifications of the arguments outlined.
 
\begin{proof}[ of Lemma \ref{lem:exist}]
The existence of $\theta(A,r)$ follows from non-negativity of $A(.)$ along with Slepian's Lemma. Fixing a positive integer $k$, Slepian's Lemma also gives that
$$\P(\sup_{[0,T]}Z(t)<r)\ge \P(\sup_{[0,1/k]}Z(t)<r)^{\lceil T\rceil k},$$ which after taking logarithms, dividing by $T$, and taking the limit $T\to\infty$ gives 
$\theta(A,r)\le -k\log \P(\sup_{[0,1/k]}Z(t)<r),$
from which we get $\theta(A,r)<\infty$ if $\P(\sup_{[0,1/k]}Z(t)<r)>0$ for some $k$. If there exists no such $k\ge 1$, then $\P(\sup_{[0,1/k]}Z(t)<r)=0$ for all $k\ge 1$, which on taking limits as $k\to\infty$ and using continuity of sample paths gives $\P(Z(0)<r)=0$, which is a contradiction as $Z(0)$ is a centered Gaussian.


{\bf Proof of (a).} 
Fix a positive integer $M$, and set $s_i:=(M+1)i$ for $i\ge 1$. Thus with $N:=\lfloor \frac{T}{M+1}\rfloor$ we have
\begin{align} \label{eqn:20-02-20hash}
\P(\sup_{t\in [0,T]}Z(t)<r)\le& \P(\max_{1\le i\le N}\sup_{[\fa{s_i-1},s_i]}Z(t)<r)
\le \P\Big(\max_{1\le i\le N}\int\limits_{\fa{s_i-1}}^{s_i}Z(t)dt<r\Big).
\end{align}
Let us denote $\zeta:=\left(\int\limits_0^{1}\int\limits_0^{1}A(u-v)dudv\right)^{1/2}$. Setting $X(i):= \zeta^{-1}\int\limits_{\fa{s_i-1}}^{s_{i}}Z(u)du$ for $i\ge 1$, the last term in (\ref{eqn:20-02-20hash}) is same as $\P(\max_{1\le i\le N}X(i)<\zeta^{-1}r)$. The covariance matrix of the centered Gaussian vector $(X(1),\ldots,X(N))$ is given by the matrix $B$, where $B(i,i):=1$, and for $i< j$,
\begin{align*}
\fa{B(i,j):=\zeta^{-2} \int\limits_0^1 \int\limits_{s_j-s_i}^{s_j-s_i+1}A(u-v)dudv=\zeta^{-2} \int\limits_0^1 \int\limits_{s_{j-i}}^{s_{j-i}+1}A(u-v)dudv}.
\end{align*}
\fa{Noting that $s_1\ge M+1$,} this can be estimated as follows:
$$
\max_{1\le i\le N}\sum_{j\ne i}B(i,j)\le 
2\zeta^{-2}\int_0^1\int_{M+1}^\infty A(u-v)du dv
\le 2\zeta^{-2} \int\limits_{M}^\infty A(u)du =:\varepsilon_M,
$$
where $\lim\limits_{M\to\infty}\varepsilon_M=0$, as $\int\limits_0^\infty A(t)dt<\infty$. 
Invoking Gershgorin's circle theorem (\cite[Thm 6.1.1]{HJ}) all eigenvalues of $B$ lie within $[1-\varepsilon_M,1+\varepsilon_M]$, and so we have
\begin{align*}
\P(\max_{1\le i\le N}X(i)<\zeta^{-1}r)=&\frac{\int\limits_{(-\infty,\zeta^{-1}r)^N}e^{-{\bf x}'B^{-1}{\bf x}/2}d{\bf x}}{\int\limits_{(-\infty,\infty)^N}e^{-{\bf x}'B^{-1}{\bf x}/2}d{\bf x}}
\le \Big(\frac{1+\varepsilon_M}{1-\varepsilon_M}\Big)^{N/2}\P\Big(\mathcal{N}(0,1)<\frac{\zeta^{-1}r}{\sqrt{1-\varepsilon_M}}\Big)^N.
\end{align*}
Combining this with (\ref{eqn:20-02-20hash}) we have the upper bound 
$\P(\sup_{t\in[0,T]}Z(t)<r)\le \beta_M^N$ with $$\beta_M:=\sqrt{\frac{1+\varepsilon_M}{1-\varepsilon_M}}\,\P\Big(\mathcal{N}(0,1)<\frac{\zeta^{-1}r}{\sqrt{1-\varepsilon_M}}\Big)\stackrel{M\rightarrow\infty}\rightarrow\P(\mathcal{N}(0,1)<\zeta^{-1}r)<1.$$ Thus there exists an $M$ such that $\beta_M<1$, and so
$\theta(A,r)\ge -\frac{1}{M+1}\log \beta_M>0$,
which completes the proof of part (a).


{\bf Proof of (b).} To begin setting $I(T):=\int_0^T A(s)ds$ we claim the existence of a  sequence of increasing positive reals $\{T_k\}_{k\ge 1}$ such that  
\begin{align}\label{eq:new_claim}
\lim_{k\to\infty}\frac{I(T_k)}{I(T_k/2)^2}=0.
\end{align}
{\it Step 1: We complete the proof of the lemma assuming (\ref{eq:new_claim}) holds.} To this effect, with $T=T_k$ and setting $Y_T:=\int_0^T Z(t)dt$ we have
$$
\sigma_T^2:=\Var(Y_T)=\int_0^T\int_0^T A(u-v)dudv=2\int_0^T(T-u) A(u)du\le 2TI(T);
$$ giving
\begin{align}
\notag\P(\sup_{t\in[0,T]}Z(t)<r)& \ge  \P(Y_T<-\delta\sigma_T\sqrt{T},\sup_{t\in[0,T]}Z(t)<r)\\
\label{eq:main1}&=\E \Big[\P\Big(\sup_{t\in[0,T]}Z(t)<r\Big|Y_T\Big){\bf 1}\Big\{Y_T\le -\delta\sigma_T\sqrt{T}\Big\}\Big].
\end{align}
\fa{We now claim that given $Y_T=y$, $\{Z(t),t\geq 0 \}$ is a Gaussian process with mean $m(t)y$ and covariance $C(t_1,t_2)$, where
\begin{align}
\begin{split}
m(t)y&:=\frac{\int\limits_0^T A(t-u)du}{\sigma_T^2}y ,\\
C(t_1,t_2)&:=A(t_1-t_2)-\frac{\int_0^T\int_0^T A(t_1-s_1)A(t_2-s_2)ds_1ds_2}{\sigma^2_T}. \label{eqn:20-02-20plus}
\end{split}
\end{align}
To derive the conditional mean in \eqref{eqn:20-02-20plus}, fixing $t\in [0,T]$ the joint distribution of $(Z(t),Y_T)$ is a centered Gaussian vector with covariance matrix
\[\begin{bmatrix}
1 &  \int_0^T A(t-u)du\\
\\
 \int_0^T A(t-u)du &\sigma^2_T
 \end{bmatrix}.\]
Thus the conditional mean of $Z(t)$ given $Y_T=y$ is given by
$$ 0+ \frac{ \int_0^T A(t-u)du}{\sigma^2_T}y=\frac{ \int_0^T A(t-u)du}{\sigma_T^2}y,$$
as claimed in \eqref{eqn:20-02-20plus}.
Focusing on the conditional covariance in \eqref{eqn:20-02-20plus}, fixing $t_1,t_2\in [0,T]$ 
 note that
$[Z(t_1),Z(t_2),Y_T]$ is a centered Gaussian vector with covariance matrix
\[\begin{bmatrix}
1 & A(t _1-t_2)&  \int_0^T A(t_1-u)du\\
\\
A(t_1-t_2) & 1 & \int_0^T A(t_2-u)du\\
\\
 \int_0^T A(t_1-u)du& \int_0^T A(t_2-u)du&\sigma_T^2
\end{bmatrix}.\]
Thus the conditional covariance matrix of $(Z(t_1), Z(t_2))$ given $Y_T=y$ is given by
{\footnotesize\[\begin{bmatrix}
1 & A(t _1-t_2)\\
\\
A(t_1-t_2) & 1
\end{bmatrix}-\frac{1}{\sigma_T^2}
\begin{bmatrix}
\int_0^T A(t_1-u)du\\
\\
\int_0^T A(t_2-u)du
\end{bmatrix}
\begin{bmatrix}
\int_0^T A(t_1-u)du&\int_0^T A(t_2-u)du
\end{bmatrix}.
\]}
 Also, the covariance term (which is the off-diagonal term above) equals
$$A(t_1-t_2)-\frac{\int_0^T A(t_1-u)du \int_0^T A(t_2-v) dv}{\sigma_T^2},$$
thus verifying \eqref{eqn:20-02-20plus}. On the set $\{Y_T\le -\delta \sigma_T \sqrt{T} \}$, we have
\begin{align*}
m(t)y=\frac{\int\limits_0^T A(t-u)du}{\sigma_T^2}y  \le  -\frac{\delta I(T/2)\sqrt{T}}{\sigma_T}\le - \frac{\delta I(T/2)}{\sqrt{2I(T)}}=:-K_T, \qquad t\geq 0.
\end{align*}
We note that due to (\ref{eq:new_claim}) we have $K_T\to\infty$ along the subsequence mentioned there. Therefore, we can assume that $K_T+r>0$.}
This gives, on the set $\{Y_T\le -\delta \sigma_T\sqrt{T}\}$,
\begin{align}
\notag\P(\sup_{t\in [0,T]}Z(t)<r|Y_T)&\ge \P\Big(\sup_{t\in [0,T]}\{|Z(t)-m(t)Y_T|\}<r+K_T|Y_T\Big)\\
\label{eq:gaussian_correlation}& \ge \prod_{i=1}^{\lceil T\rceil}\P\Big(\sup_{t\in [i-1,i]}\{|Z(t)-m(t)Y_T|\}<r+K_T|Y_T\Big),
\end{align}
where the second inequality is by the Gaussian correlation inequality (\cite{royen}, also see \cite[Theorem 1]{Lat_Mal}).
Proceeding to estimate the right hand side above, first note that by (\ref{eqn:20-02-20plus}) and non-negative correlations, we have
\begin{align*}
\Var\Big[Z(t)-Z(s)\Big|Y_T\Big]\le \Var(Z(t)-Z(s)),
\end{align*}
which along with Sudakov-Fernique inequality (\cite[Theorem 2.2.3]{AT}) gives
\begin{equation}
\E\Big( \sup_{t\in [i-1,i]}\{Z(t)-m(t)Y_T\}|Y_T\Big)\le \E \sup_{t\in [i-1,i]}Z(t)= \E \sup_{t\in [0,1]}Z(t)=:\alpha<\infty,
\label{eqn:newref1}
\end{equation}
using the Borel-TIS inequality (\cite[Thm 2.1.1]{AT}). Invoking Borell-TIS inequality we now get 
\begin{align*}
\P(\sup_{t\in [i-1,i]}\{|Z(t)-m(t)Y_T|\}>r+K_T|Y_T\Big)
& = 2\P\Big(\sup_{t\in [i-1,i]}\{Z(t)-m(t)Y_T\}>r+K_T|Y_T\Big)\\
& \le  2 e^{-\frac{1}{2}(r+K_T-\alpha)^2},
\end{align*}
where we used (\ref{eqn:newref1}) and that fact that $K_T+r-\alpha>0$ along the subsequence mentioned in (\ref{eq:new_claim}), because $K_T\to\infty$ along that subsequence. The last relation, along with \eqref{eq:gaussian_correlation} gives
$$\P(\sup_{t\in [0,T]}Z(t)<r|Y_T)\ge \Big[1-2e^{-\frac{1}{2}(r+K_T-\alpha)^2}\Big]^{\lceil T\rceil}.$$
Combining with \eqref{eq:main1}, this gives
$$\P(\sup_{t\in [0,T]}Z(t)<r)\ge  \Big[1-2e^{-\frac{1}{2}(r+K_T-\alpha)^2}\Big]^{\lceil T\rceil}\P(Y_T\le -\delta \sigma_T \sqrt{T}),$$
which on taking $\log$, dividing by $T$, and letting $T\to\infty$ along the sequence $\{T_k\}_{k\ge 1}$ gives
$-\theta(A,r)\ge -\frac{\delta^2}{2},$
where we have used \eqref{eq:new_claim} to conclude that that $K_T\to\infty$. The desired conclusion then follows since $\delta>0$ is arbitrary.

{\it Step 2: We prove  \eqref{eq:new_claim}.} Assume by way of contradiction that \eqref{eq:new_claim} does not hold, which implies $\liminf_{T\to\infty}\frac{I(T)}{I(T/2)^2}>0.$ Thus there exists $\varepsilon>0$ and $T_0>0$ such that  we have
\begin{equation} \label{eqn:20-02-20-iterate}
I(T)\ge \varepsilon I(T/2)^2,\qquad \forall T\ge 2T_0.
\end{equation}

Using the assumption that $I$ diverges, we can now make $T_0$ even larger such that $\eps I(T_0)>2$. Further, by iterating (\ref{eqn:20-02-20-iterate}), we obtain
$$
I(2^k T_0)\ge \varepsilon^{1+2+2^2+2^{k-1}}  I(T_0)^{2^k} \geq (\eps I(T_0))^{2^k} ,
$$
which along with the trivial bound $I(T)\le T$ and the choice of $T_0$ (namely, $\eps I(T_0)\geq 2$) gives $2^kT_0\ge 2^{2^k}$.
But this is a contradiction, as the right hand side grows much faster than the left hand side as $k\to\infty$. This completes the proof of \eqref{eq:new_claim}.
\end{proof}

\fa{\begin{proof}[ of Theorem \ref{lem:ohad}]
If $\int_0^\infty A(t)dt=\infty$, then using Lemma \ref{lem:exist} part (ii) we have
$\theta(A,r)=0$ for every $r\in\mathbb{R}$, and so continuity is trivial. So without loss of generality assume that $\int_0^\infty A(t)dt<\infty$.
Use continuity of sample paths to note that
$$\P(\sup_{t\in [0,M]}Z(t)=r)=\P(\sup_{t\in [0,M]\cap\Q}Z(t)=r)\le \sum_{q\in [0,M]\cap \Q}\P(Z(q)=r)=0,$$
and so the distribution function of $\sup_{t\in [0,M]}Z(t)$ is continuous. 
For showing left continuity of $r\mapsto \theta(A,r)$, 
fixing $\eta<0, M>0$ using stationary and non-negativity of the correlation function along with Slepian's Lemma gives
$$\P(\sup_{t\in [0,T]}Z(t)<r+\eta)\ge \P(\sup_{t\in [0,M]}Z(t)<r+\eta)^{\lceil \frac{T}{M}\rceil},$$
which on taking $\log$, dividing by $T$ and letting $T\to\infty$ gives
$$-\theta(A,r+\eta)\ge \frac{1}{M}\log \P(\sup_{t\in [0,M]}Z(t)<r+\eta).$$
On letting $\eta\to 0$ this gives
$$\limsup_{\eta\uparrow 0}\theta(A,r+\eta)\le -\frac{1}{M}\log \P(\sup_{t\in [0,M]}Z(t)<r),$$
where we use the fact that the distribution function of $\sup_{t\in [0,M[}Z(t)$ is continuous.
Letting $M\to\infty$ then gives
$$\limsup_{\eta\uparrow0}\theta(A,r+\eta)\le \theta(A,r),$$
and so we have verified left continuity.
\\
\\
We now proceed to verify right continuity, which is the main contribution of this theorem. To this effect, recall from \eqref{eqn:20-02-20plus} in the proof of part (b) of Lemma \ref{lem:exist}  that given $Y_T=\int_0^TZ(t)dt=y$, the conditional distribution of $\{Z(t),t\ge 0\}$ is a Gaussian process with mean $$m(t)y=\frac{\int_0^TA(t-u)du}{\sigma_T^2}y,$$
and covariance
{\footnotesize \begin{align*}
C(t_1,t_2):=&A(t_1-t_2)-\frac{\int_0^T\int_0^T A(t_1-s_1)A(t_2-s_2)ds_1ds_2}{\sigma_T^2}, \quad \sigma_T^2:=\Var(Y_T)=2\int_0^T (T-s)A(s)ds.
\end{align*}}
Letting $\bar{Z}_{\bar{m}(t)}$ denote a Gaussian process on $[0,T]$ with mean function $\bar{m}(.)$ and covariance $C(.,.)$ fixing $K<\infty$ we have
\begin{align*}
\P(\sup_{t\in [0,T]}Z(t)<r+\eta,Y_T>-KT)=&\frac{1}{\sqrt{2\pi \sigma_T^2}}\int_{-KT}^{\infty} \P(\sup_{t\in [0,T]}\bar{Z}_{m(t)y}<r+\eta) e^{-\frac{y^2}{2\sigma_T^2}}dy\\
=&\frac{1}{\sqrt{2\pi \sigma_T^2}}\int_{-KT}^\infty \P(\sup_{t\in [0,T]}\bar{Z}_{m(t)y-\eta}<r) e^{-\frac{y^2}{2\sigma_T^2}}dy.
\end{align*}
Since $$m(t)=\frac{\int_0^T A(t-u)du}{2\int_0^T (T-u)A(u)du}\le \frac{2\int_0^\infty A(t)dt}{T\int_0^{T/2}A(t)dt},$$  there exists $C>0$ such that for all $T$ large enough we have $\sup_{t\in [0,T]}m(t)\le \frac{C}{T}$, and consequently 
 $$ \P(\sup_{t\in [0,T]}\bar{Z}_{m(t)y-\eta}<r)\le \P(\sup_{t\in [0,T]}\bar{Z}_{m(t)(y-\frac{\eta T}{C})}<r),$$ which gives
\begin{align*}
\P(\sup_{t\in [0,T]}Z(t)<r+\eta, Y_T>-KT))\le \frac{1}{\sqrt{2\pi \sigma_T^2}}\int_{-KT}^\infty \P(\sup_{t\in [0,T]}\bar{Z}_{m(t)(y-\frac{\eta T}{C})}<r) e^{-\frac{y^2}{2\sigma_T^2}}dy.
\end{align*}
Changing variables to $y'=y-\frac{\eta T}{C}$ we get
\begin{align*}
\P(\sup_{t\in [0,T]}Z(t)<r+\eta,Y_T>-KT)\le &\frac{1}{\sqrt{2\pi \sigma_T^2}}\int_{-KT-\frac{\eta T}{C}}^\infty \P(\sup_{t\in [0,T]}\bar{Z}_{m(t)y'}<r) e^{-\frac{(y'+\frac{\eta T}{C})^2}{2\sigma_T^2}}dy'\\
\le &e^{\frac{\eta T^2}{C\sigma_T^2}(K+\frac{\eta }{C})}\frac{1}{\sqrt{2\pi \sigma_T^2}}\int_{-\infty}^\infty \P(\sup_{t\in [0,T]}\bar{Z}_{m(t)y'}<r)e^{-\frac{y'^2}{2\sigma_T^2}}dy'\\
=&e^{\frac{\eta T^2}{C\sigma_T^2}(K+\frac{\eta }{C})}\P(\sup_{t\in [0,T]}Z(t)<r).
\end{align*}
This gives
\begin{align*}
\P(\sup_{t\in [0,T]}Z(t)<r+\eta)\le &\P(\sup_{t\in [0,T]}Z(t)<r+\eta, Y_T>-KT)+\P(Y_T<-KT)\\
\le& e^{\frac{\eta T^2}{C\sigma_T^2}(K+\frac{\eta }{C})}\P(\sup_{t\in [0,T]}Z(t)<r)-\P(Y_T<-KT),
\end{align*}
or equivalently,
\begin{align}\label{eq:rev}
\P(\sup_{t\in [0,T]}Z(t)<r+\eta)+\P(Y_T<-KT) 
\le& e^{\frac{\eta T^2}{C\sigma_T^2}(K+\frac{\eta }{C})}\P(\sup_{t\in [0,T]}Z(t)<r).
\end{align}
An application of monotone convergence theorem gives $$\lim_{T\to\infty}\frac{\sigma_T^2}{T}=2\int_0^T \Big(1-\frac{s}{T}\Big)A(s)ds\to2 \int_0^\infty A(s)ds=: \lambda>0.$$ Thus, on taking $\log$, dividing by $T$, and letting $T\to\infty$ in \eqref{eq:rev} gives
$$\max\Big\{-\theta(A,r+\eta),-\frac{K^2}{2\lambda}\Big\}\le \frac{\eta}{C\lambda}\Big(K+\frac{\eta}{c}\Big)-\theta(A,r),$$
which on letting $\eta\downarrow 0$ followed by $K\to\infty$ gives 
$$\liminf_{\eta\downarrow 0}\theta(A,r+\eta)\ge \theta(A,r),$$
which verifies right continuity, and hence completes the proof of the theorem.
\end{proof}}

\begin{proof}[ of Lemma \ref{lem:op}] {\it Step 1: Interpolating between $Y_n(.)$ and $Z_n(.)$.} Given a measurable set $D\subseteq \R$ of positive Lebesgue measure, define 
$$
g_n(\gamma,D):=\log \int_{D^n} \text{exp}\left\{-\frac{1}{2}{\bf x}'\Big((1-\gamma)A_n^{-1}+\gamma B_n^{-1}\Big){\bf x}\right\}d{\bf x},\qquad \gamma\in[0,1],
$$ and note that
\begin{align*}
\partial_\gamma g_n(\gamma,D)=&\frac{1}{2} \E\Big(X_{n,\gamma}'(A_n^{-1}-B_n^{-1})X_{n,\gamma}\Big|X_{n,\gamma}(i)\in D, 1\le i\le n\Big),
\end{align*}
where $\{X_{n,\gamma}(i)\}_{1\le i\le n}$ is a centered Gaussian vector with inverse covariance matrix $\Sigma_{n,\gamma}^{-1}:=(1-\gamma)A_n^{-1}+\gamma B_n^{-1}$. Note that $\Sigma_{n,\gamma}$ is positive definite (because all eigenvalues are bounded away from zero, by \eqref{eq:compact}) and symmetric, so that indeed it is a proper covariance matrix of a Gaussian vector. By construction, we have $X_{n,0}(.)\stackrel{d}{=}Z_n(.)$ and $X_{n,1}(.)\stackrel{d}{=}Y_n(.)$, and so with $D_r:=(-\infty,r)$ we have
\begin{eqnarray}
\notag&&|\log \P(\max_{1\le i\le n}Y_n(i)<r)-\log \P(\max_{1\le i\le n}Z_n(i)<r)|\\
\notag&=&|g_n(1,D_r)-g_n(1,D_\infty)-g_n(0,D_r)+g_n(0,D_\infty)|\\
\notag&\le& \sup_{\gamma\in [0,1]}|\partial_\gamma g_n(\gamma,D_r)|+\sup_{\gamma\in [0,1]}|\partial_\gamma g_n(\gamma,D_\infty)|\\
\notag&\le& \frac{1}{2}||A_n^{-1}-B_n^{-1}||_2 \sup_{\gamma\in [0,1]}\E(\sum_{i=1}^nX_{n,\gamma}(i)^2|X_{n,\gamma}(i)\in D_r,1\le i\le n)\\
\label{eq:bound_nov0}&&+   \frac{1}{2}||A_n^{-1}-B_n^{-1}||_2   \sup_{\gamma\in [0,1]}\E(\sum_{i=1}^nX_{n,\gamma}(i)^2|X_{n,\gamma}(i)\in D_\infty,1\le i\le n).
\end{eqnarray}
We now claim that for any $D$ we have
\begin{align}\label{eq:bound_nov}
\limsup_{n\to\infty}\sup_{\gamma\in [0,1]}\frac{1}{n}\E (\sum_{i=1}^nX_{n,\gamma}(i)^2|X_{n,\gamma}(i)\in D,1\le i\le n)<\infty.
\end{align}
Deferring the proof of \eqref{eq:bound_nov}, use \eqref{eq:compact} along with $||A_n-B_n||_2\to0$ to conclude that $||A_n^{-1}-B_n^{-1}||_2\to 0$ (use that $A_n^{-1}-B_n^{-1}=A_n^{-1}(B_n-A_n)B_n^{-1}$). This observation along with \eqref{eq:bound_nov0} and \eqref{eq:bound_nov} gives the desired conclusion.

{\it Step 2:  We show \eqref{eq:bound_nov}.} First, we use \eqref{eq:compact} to get the existence of positive reals $\lambda_1,\lambda_2$ free of $n$ such that all eigenvalues of $A_n^{-1}$ and $B_n^{-1}$ lie within the interval $[\lambda_1,\lambda_2]$. In particular this implies that the eigenvalues of $\Sigma_{n,\gamma}:=(1-\gamma)A_n^{-1}+\gamma B_n^{-1}$ also lie in the same interval, and so fixing $K>0$ we have
\begin{align*}
\notag&\E \Big(\sum_{i=1}^nX_{n,\gamma}(i)^2|X_{n,\gamma}(i)\in D,1\le i\le n\Big)\\
\notag\le &Kn+\E \Big(\sum_{i=1}^nX_{n,\gamma}(i)^2{\bf 1}\Big\{\sum_{i=1}^nX_{n,\gamma}(i)^2>Kn\Big\}\Big|X_{n,\gamma}(i)\in D,1\le i\le n\Big)\\
\notag\le &Kn+\frac{\E \Big(\sum_{i=1}^nX_{n,\gamma}(i)^2 {\bf 1}\Big\{\sum_{i=1}^nX_{n,\gamma}(i)^2>Kn\Big\}\Big)}{\P\Big(X_{n,\gamma}(i)\in D,1\le i\le n\Big)}\\
\le &Kn+\frac{\sqrt{\E\Big[ \Big(\sum_{i=1}^nX_{n,\gamma}(i)^2\Big)^2\Big]}\sqrt{\P\Big(\sum_{i=1}^nX_{n,\gamma}(i)^2>Kn\Big)}}{\P\Big(X_{n,\gamma}(i)\in D,1\le i\le n\Big)},
\end{align*}
and so it suffices to show that
\begin{align}\label{eq:nov2}
\sup_{\gamma\in [0,1]} \frac{\sqrt{\E\Big[ \Big(\sum_{i=1}^nX_{n,\gamma}(i)^2\Big)^2\Big]}\sqrt{\P\Big(\sum_{i=1}^nX_{n,\gamma}(i)^2>Kn\Big)}}{\P\Big(X_{n,\gamma}(i)\in D,1\le i\le n\Big)}\leq e^{-C(K) n},
\end{align}
with $C(K)>0$ for $K$ large enough.

We will now bound each of the terms in the left hand side of \eqref{eq:nov2}. To begin, note that 
\begin{eqnarray}
\P\Big(\sum_{i=1}^nX_{n,\gamma}(i)^2>Kn\Big)&=&\frac{\int_{\{\sum_{i=1}^nx_i^2>Kn\}} e^{-\frac{1}{2}{\bf x}'\Sigma_{n,\gamma}^{-1}{\bf x}}d{\bf x}}{\int_{\R^n} e^{-\frac{1}{2}{\bf x}'\Sigma_{n,\gamma}^{-1}{\bf x}}d{\bf x}}
\notag\\
&\le &\frac{\int_{\{\sum_{i=1}^nx_i^2>Kn\}} e^{-\frac{\lambda_1}{2}{\bf x}'{\bf x}}d{\bf x}}{\int_{\R^n} e^{-\frac{\lambda_2}{2}{\bf x}'{\bf x}}d{\bf x}}
\notag\\
&=&\Big(\frac{\lambda_2}{\lambda_1}\Big)^{\frac{n}{2}}\P(\sum_{i=1}^n \xi_i^2>\lambda_1 Kn)
\notag
\\
&\le & \Big(\frac{\lambda_2}{\lambda_1}\Big)^{\frac{n}{2}} \Big( \E[ e^{\xi^2/4}]\Big)^n e^{-\lambda_1 K n / 4}
\notag
\\
&\le & e^{-\tilde C(K) n},\label{eq:K}
\end{eqnarray}
where the $(\xi_i)$ are i.i.d.\ $\NN(0,1)$ and where $\tilde C(K)$ can be made arbitrarily large by making $K$ large. A similar calculation gives
\begin{eqnarray}
\P\Big(X_{n,\gamma}(i)\in D,1\le i\le n\Big)&=&\frac{\int_{D^n} e^{-\frac{1}{2}{\bf x}'\Sigma_{n,\gamma}^{-1}{\bf x}}d{\bf x}}{\int_{\R^n} e^{-\frac{1}{2}{\bf x}'\Sigma_{n,\gamma}^{-1}{\bf x}}d{\bf x}}
\notag
\\
&\ge &\frac{\int_{D^n} e^{-\frac{\lambda_2}{2}{\bf x}'{\bf x}}d{\bf x}}{\int_{\R^n} e^{-\frac{\lambda_1}{2}{\bf x}'{\bf x}}d{\bf x}}
\notag
\\
&=&\Big(\frac{\lambda_1}{\lambda_2}\Big)^{\frac{n}{2}}\P(\mathcal{N}(0,1)\in \sqrt{\lambda_2} D)^n, \label{eq:D}
\end{eqnarray}
which is at most an exponentially decreasing factor. 
Finally we have using the Cauchy-Schwartz inequality
\begin{align}\label{eq:E}
\E \Big[\Big(\sum_{i=1}^nX_{n,\gamma}(i)^2\Big)^2\Big]\le n \sum_{i=1}^n \E X_{n,\gamma}(i)^4\le \frac{3n^2}{\lambda_1^2},
\end{align}
which increases polynomially in $n$. Combining \eqref{eq:K}, \eqref{eq:D}, \eqref{eq:E} and making $K$ sufficiently large to compensate the other exponential factors gives \eqref{eq:nov2}, and hence completes the proof of the lemma.
\end{proof}

\begin{proof}[ of Lemma \ref{lem:DM2d}] To begin, use the uniform convergence of correlation functions to conclude that for any positive integer $L$ we have
\begin{align}\label{eq:unif_2}
\xi_{k,L}:=\sup_{s\in \N_0}\sup_{(x_1,\ldots,x_L)\in \R^L}\Big|\P(Z_k(s+i)<x_i,i \le L)-\P(Z_\infty(i)<x_i, i \le L)\big|\stackrel{k\to\infty}{\to}0.
\end{align}

{\it Step 1: Lower bound.}  Fix a positive integer $M$ and use Slepian's lemma and non-negative correlations gives
\begin{align*}
\P(\max_{1\le i\le n}Z_k(i)<r)\ge &\prod_{a=1}^{\lceil \frac{n}{M}\rceil}\P(\max_{1\le i\le M}Z_k((a-1)M+i)<r)\\
\ge &\Big[\P(\max_{1\le i\le M} Z_\infty(i)< r)-\xi_{k,M}\Big]^{\lceil \frac{n}{M}\rceil}, 
\end{align*}
where the last inequality uses \eqref{eq:unif_2} with $L=M$. Taking $\log$, dividing by $n$, and letting $n\to\infty$ on both sides of the above equation gives 
\begin{equation}
\label{eqn:20200228star}
\liminf_{k,n\to\infty}\frac{1}{n}\log \P(\sup_{1\le i\le n}Z_k(i)<r) \ge  \frac{1}{M}\log \P(\sup_{1\le i\le M}Z_\infty(i)<r).
\end{equation}
 The lower bound of \eqref{eq:conclusion_d} follows from this on letting $M\to\infty$. 

{\it Step 2: Introducing a perturbation.}
For the upper bound of \eqref{eq:conclusion_d}, 
let $W_1,\ldots,W_n$ be i.i.d.\ $\mathcal{N}(0,1)$, independent of $\{Z_k(t)\}_{t\ge 0}$. Then for any $h>0$ we have, by the independence of the $\{W_i\}$ and $\{Z_k(t)\}$,
\begin{align*}
\P(\max_{1\le i\le n}Z_k(i)<r)\
& \le \frac{\P(\max_{1\le i\le n}\{Z_k(i)+hW_i\}<r+\sqrt{h})}{\P(W_1 < h^{-1/2})^n},
\end{align*}
which on taking $\log$, dividing by $n$, and letting $n,k\to\infty$ gives
\begin{eqnarray}\label{eq:upper_1}
\notag\limsup_{k,n\to\infty}\frac{1}{n}\log \P(\sup_{1\le i\le n}Z_k(i)<r)&\le& \limsup_{k,n\to\infty}\frac{1}{n}\log \P(\max_{1\le i\le n}\{Z_k(i)+hW_i\}<r+\sqrt{h})\\
&&-\log \P(W_1 < h^{-1/2}).
\end{eqnarray}

{\it Step 3: Introducing blocks.}
Fix positive integers $M,m$ with $M>m$, and let $p_i:=i(m+M)$ for $i\ge 1$. Let $N:=\lfloor \frac{n}{M+m}\rfloor$ denote the largest integer $i$ such that $p_i\le n$. Set $I_a:=[p_a-M,p_a]$ for $1\le a\le N$, and note the trivial inequality:
 \begin{align}\label{eq:trivial}
 \P(\max_{1\le i\le n}\{Z_k(i)+hW(i)\}<r+\sqrt{h})\le \P(\max_{1\le a\le N}\max_{i\in  I_a}\{Z_k(i)+hW(i)\}<r+\sqrt{h}).
 \end{align}
Let  $\{Y_k(i)\}_{i\in I_a, 1\le a\le N}$ be a centered Gaussian vector such that the collection of vectors $(\{Y_k(i)\}_{i\in I_a}, 1\le a\le N)$ are mutually independent, with
$$\{Y_k(i)\}_{i\in I_a}\stackrel{d}{=}\{Z_k(i)\}_{i\in I_a}.$$ Denoting $A_{n,k}$ and $B_{n,k,M,m}$ to be the covariance matrices of $\{Z_k(i)+hW(i)\}_{i\in I_a, 1\le a\le N}$ and $\{Y_k(i)+h W(i)\}_{i\in  I_a, 1\le a\le N}$, respectively, we have
\begin{align*}
\max_{i\in \bigcup_{a=1}^N I_a}\sum_{b=1}^N\sum_{j\in  I_b}|A_{n,k}(i,j)-B_{n,k,M,m}(i,j)| 
&
=\max_{1\le a\le N}\max_{i\in I_a}\sum_{b\ne a}^N\sum_{j\in  I_b }A_k(i,j)
\\
&\le \max_{1\le i\le n} \sum_{j:|j-i|\ge m}A_k(i,j)
\\
&\le 2\sum_{i=m }^\infty g_k(i)=:\varepsilon_{m,k},
\end{align*}
where $\lim_{m\to\infty}\lim_{k\to\infty}\varepsilon_{m,k}=0$ using \eqref{eq:DM1d}. Invoking the Gershgorin circle theorem (\cite[Thm 6.1.1]{HJ}), this gives
\begin{align}\label{eq:op1}
\limsup_{m\to\infty}\limsup_{M\to\infty}\limsup_{k,n\to\infty}||A_{n,k}-B_{n,k,M,m}||_2=0.
\end{align}
Also note that
\begin{align*}
\limsup_{k,n\rightarrow\infty}\max_{1\le i\le n}\sum_{j=1}^nA_{n,k}(i,j)\le2 \limsup_{k\rightarrow\infty}\sum_{\tau=0}^\infty g_k(\tau),
\end{align*}
which is finite using
\eqref{eq:DM1d}, and consequently we have
\begin{align}\label{eq:op2}
\limsup_{k,n\to\infty}||A_{n,k}||_2+\limsup_{k,n\to\infty}||B_{n,k,M,m}||_2<\infty.
\end{align}
Given that the eigenvalues of $A_{n,k}$ and $B_{n,k,M,m}$ are bounded below by $h>0$, we see that \eqref{eq:compact} is satisfied. Therefore, with \eqref{eq:op1} and \eqref{eq:op2}, an application of Lemma \ref{lem:op} gives
\begin{align}\label{eq:upper2}
\notag\limsup_{m\to\infty}\limsup_{M\to\infty}\limsup_{k,n\to\infty}\frac{1}{n}&\Big|\log \P(\max_{1\le a\le N}\max_{i\in  I_a}\{Z_k(i)+hW(i)\}<r+\sqrt{h})\\
-&\log \P(\max_{1\le a\le N}\max_{i\in  I_a}\{Y_k(i)+hW(i)\}<r+\sqrt{h})\Big|=0.
\end{align}

{\it Step 4: Decoupling the blocks.}  By definition of $\{Y_k(.)\}$ we have
\begin{align*}
&\notag\P(\max_{1\le a\le N}\max_{i\in  I_a}\{Y_k(i)+hW(i)\}<r+\sqrt{h})\\
&=\prod_{a=1}^N\P(\max_{i\in  I_a}\{Z_k(i)+hW(i)\}<r+\sqrt{h})\\
\notag
&\le \Big[\P(\max_{1\le i\le M}\{Z_\infty(i)+hW(i)\}<r+\sqrt{h})+\xi_{k,M}\Big]^N,
\end{align*}
where the last inequality uses \eqref{eq:unif_2} with $L=M$. On taking $\log$, dividing by $n$, and letting $n,k\to\infty$ on both sides of the above equation we get
\begin{align}\label{eq:upper3}
\notag\limsup_{k,n\to\infty}&\frac{1}{n}\log \P(\max_{1\le a\le N}\max_{i\in  I_a}\{Y_k(i)+hW(i)\}<r+\sqrt{h})\\
\le& \frac{1}{M+m}\log \P(\max_{1\le i\le M }\{Z_\infty(i)+hW(i)\}<r+\sqrt{h}).
\end{align}
{\it Step 5: Undoing the perturbation.}
For analyzing the right hand side of \eqref{eq:upper3}, set $S:=\{i\in \{1,\ldots,M\}:|W(i)|>h^{-1/2}\}$. Fix $0<t<1/2$. Then, denoting the set of all $p$ ordered tuples of distinct integers in $\{1,\ldots,M\}$ by $[M]_p$ we have
\begin{align}
\notag&\P(\max_{1\le i\le M}\{Z_\infty(i/\ell)+hW(i)\}<r+\sqrt{h})\\
\notag&\le \P(|S|>Mt)+\P(|S|\le Mt, \max_{1\le i\le M}\{Z_\infty(i)+hW(i)\}<r+\sqrt{h})\\
\notag&\le \P(|S|>Mt)+\sum_{p=\lceil M(1-t) \rceil}^M\sum_{(i_1,\cdots,i_p)\in [M]_p}\P(\max_{1\le j\le p}\notag Z_\infty(i_j)<r+2\sqrt{h})\\
\notag&\le \P(|S|>Mt)+\sum_{p=\lceil M(1-t) \rceil}^M\sum_{(i_1,\cdots,i_p)\in [M]_p}\frac{\P(\max_{1\le i\le M}Z_\infty(i)<r+2\sqrt{h})}{\P(Z_\infty(0)<r+2\sqrt{h})^{M-p}}\\
\label{eq:slepian_1}
&\le \P(|S|>Mt)+M{M\choose \lceil M(1-t) \rceil}\frac{\P(\max_{1\le i\le M}Z_\infty(i)<r+2\sqrt{h})}{\P(Z(0) < r+2\sqrt{h})^{Mt}},
\end{align}
where the inequality in the penultimate line uses Slepian's inequality along with non-negative correlations. We will now analyze both the terms in the right hand side of \eqref{eq:slepian_1} separately. 
Since $|S|\sim {\rm Bin}(M,\P(|W_1|>h^{-1/2}))$, for every $t>0$ we have
\begin{align}\label{eq:slepian_2}
\limsup_{h\to0}\limsup_{M\to\infty}\frac{1}{M}\log \P(|S|>Mt)=-\infty.
\end{align}
Also using the asymptotics of binomial coefficients gives $$\lim_{M\to\infty}\frac{\log {M\choose \lceil M(1-t)\rceil}}{M}=-t\log t-(1-t)\log (1-t),$$ so that
\begin{align}\label{eq:slepian_3}
\notag\limsup_{t\to0}&\limsup_{M\to\infty}\frac{1}{M}\log \left\{M{M\choose \lceil M(1-t) \rceil}\frac{\P(\max_{1\le i\le M}Z_\infty(i)<r+2\sqrt{h})}{\P(Z_\infty(0) < r+2\sqrt{h})^{Mt}}\right\}\\
\le&  \lim_{M\to\infty}\frac{1}{M}\log \P(\max_{1\le i\le M}Z_\infty(i)<r+2\sqrt{h}).
\end{align}
Combining \eqref{eq:upper_1}, \eqref{eq:trivial}, \eqref{eq:upper2}, \eqref{eq:upper3}, \eqref{eq:slepian_1}, \eqref{eq:slepian_2}, and \eqref{eq:slepian_3} and taking limits in the order $n,k\to\infty$, $M\to\infty$, $m\to\infty$, $h\to 0$ and then $t\to 0$ gives
$$
\limsup_{k,n\to\infty}\frac{1}{n}\log \P(\sup_{1\leq i\leq n}Z_k(i)<r)
\le \lim_{h\to 0}\lim_{M\to\infty}\frac{1}{M}\log \P(\max_{1\le i\le M}Z_\infty(i)<r+2\sqrt{h}),
$$
from which the upper bound follows using Theorem \ref{lem:ohad}.
\end{proof}

\begin{proof}[ of Theorem~\ref{lem:DM2}] Under the assumption \eqref{eq:DM3}, uniform convergence of correlation functions gives

\begin{align}\label{eq:unif_2c}
\xi_{k,L}:=\sup_{s\in \N_0}\sup_{x\in \R}\Big|\P(\sup_{\tau\in [0,L]}Z_k(s+\tau)<x)-\P(\sup_{\tau\in  [0,L]}Z_\infty(\tau)<x)|\stackrel{k\to\infty}{\to}0.
\end{align}
Using this, repeating arguments similar to the proof of the lower bound in the discrete case (cf.\ (\ref{eqn:20200228star})) gives 
\begin{equation}
\label{eqn:20200228doublestar}
\liminf_{k,T\to\infty}\frac{1}{T}\log \P(\sup_{t\in [0,T]}Z_k(t)<r) \ge \frac{1}{M}\log \P(\sup_{t\in [0,M]}Z_\infty(t)<r).
\end{equation}
The lower bound of \eqref{eq:conclusion} follows from this on letting $M\to\infty$.
\\

For the upper bound of \eqref{eq:conclusion}, assume without loss of generality that $T$ is an integer. Fixing a positive integer $\ell$ and setting $n=T\ell$ we have the trivial upper bound:
$$
\P(\sup_{t\in [0,T]}Z_k(t)<r)\le  \P(\max_{1\le i\le n}Z_k(i/\ell)<r),
$$
where the correlation of the discrete time Gaussian process $\{Z_k(i/\ell)\}_{i\ge 1}$ converge to that of $\{Z_\infty(i/\ell)\}_{i\ge 1}$. Also $\{Z_k(i/\ell)\}_{i\ge 1}$ satisfies \eqref{eq:DM1d} by assumption (\ref{eq:DM1}), and so an application of Lemma \ref{lem:DM2d} gives
\begin{align}\label{eq:verbatim_2020}
\limsup_{k,T\to\infty}\frac{1}{T}\log \P(\sup_{t\in [0,T]}Z_k(t)<r)\le \limsup_{T\rightarrow\infty}\frac{1}{T}\log \P(\max_{1\le i\le T \ell}Z_\infty(i/\ell)<r).
\end{align}
From this the upper bound of \eqref{eq:conclusion} follows on letting $\ell\to \infty$ and invoking \eqref{eq:cont_1}. This completes the proof of the lemma. 
\end{proof}

\begin{proof}[ of Lemma \ref{lem:DM22}]
{\bf Proof of (a).}  This follows on noting that the lower bound proof of Theorem~\ref{lem:DM2} only uses (\ref{eq:DM3}), cf.\ (\ref{eqn:20200228doublestar}).

{\bf Proof of (b).} This follows on noting that the upper bound proof of Theorem~\ref{lem:DM2} only uses \eqref{eq:DM1} up to the step \eqref{eq:verbatim_2020}, giving 
\begin{align*}
\limsup_{k,T\to\infty}\frac{1}{T}\log \P(\sup_{t\in [0,T]}Z_k(t)<r)&\le \limsup_{T\rightarrow\infty}\frac{1}{T}\log \P(\max_{1\le i\le T \ell}Z_\infty(i/\ell)<r)\\
&=\ell \log \P(\mathcal{N}(0,1)<r),
\end{align*}
where the last equality uses the fact that the limiting process is white noise. The desired conclusion then follows on letting $\ell\to\infty$.
\end{proof}

{\bf Acknowledgement.} We thank the American Institute of Mathematics for hosting the SQuaRE ``Persistence probabilities'', where this project was initiated. We also thank Mikhail Lifshits for suggesting the Sudakov-Fernique inequality for the proof of Prop.\ 2.2.


\section{Appendix}

\begin{proof}[ of Lemma \ref{lem:f}]
Let us start with (a). If $x\in [0,1]$ and $y\in [N,b]$, then we have $$\frac{|x-y|^{2H-2}}{y^{2H-2}}=\Big(1-\frac{x}{y}\Big)^{2H-2}\le \Big(1-\frac{1}{N}\Big)^{2H-2},$$ and so 

\begin{align*}
f_{p,H}(1,b)&\le f_{p,H}(1,N)+\Big(1-\frac{1}{N}\Big)^{2H-2}\int_0^1\int_{1}^b x^p y^{p+2H-2} dydx\\
&= f_{p,H}(1,N)+\Big(1-\frac{1}{N}\Big)^{2H-2} \frac{\psi_{p+2H-2}(b)}{p+1},
\end{align*}
 which verifies \eqref{eq:f1}.

For the lower bound, note that $|y-x|^{2H-2}\ge y^{2H-2}$ for $y\ge x$, and so
\begin{align*}
f_{p,H}(1,b)-f_{p,H}(1,1)\ge\int_0^1 \int_{1}^b x^p y^{p+2H-2}dydx \ge  \frac{\psi_{p+2H-2}(b)}{p+1},
\end{align*}
which verifies \eqref{eq:f2}.

Simply observe that (b) follows from using Selberg's integral formula (see e.g.\ \cite[(1.1)]{FW}) with $\alpha=p+1, \beta=1, \gamma=H-1$. 
\end{proof}

\begin{proof}[ of Lemma \ref{lem:limit}]
For the upper bound in \eqref{eq:limit}, for any $b\ge 1$ and $N\ge 1$, for all $u$ large we have $bu\ge u\ge N$, and so
\begin{eqnarray}\label{eq:upper_bound_2018}
\notag F_{\rho,\sigma}(u,bu)& =&\int_0^u\int_0^{bu}\sigma(\lceil x\rceil)\sigma(\lceil y\rceil)\rho(\lceil x\rceil-\lceil y\rceil)dxdy\\
\notag&\le& 2\int_0^N \int_0^{bu}\sigma(\lceil x\rceil)\sigma(\lceil y\rceil)\rho(\lceil x\rceil-\lceil y\rceil)dxdy\\
\notag&&+\int_0^u\int_0^{bu}\sigma(\lceil x\rceil)\sigma(\lceil y\rceil)\rho(\lceil x\rceil-\lceil y\rceil){\bf 1}\{|x-y|\le N\}dxdy\\
&&+\int_{N}^{u}\int_{N}^{bu}\sigma(\lceil x\rceil)\sigma(\lceil y\rceil)\rho(\lceil x\rceil-\lceil y\rceil){\bf 1}\{|x-y|>N\}dxdy.
\end{eqnarray}
We will now bound each of the terms on the right hand side of \eqref{eq:upper_bound_2018}. For the first term, we have
\begin{eqnarray}\label{eq:2018_1}
&&\notag\int_0^N\int_0^{bu}\sigma(\lceil x\rceil)\sigma(\lceil y\rceil)\rho(\lceil x\rceil-\lceil y\rceil)dxdy\\
&\lesssim_\rho&\notag\int_0^N\int_0^{2N}\sigma(\lceil x\rceil)\sigma(\lceil y\rceil)\rho(\lceil x\rceil-\lceil y\rceil)dxdy\\
&&+\notag\int_0^N\int_{2N}^{bu}\sigma(\lceil x\rceil)\sigma(\lceil y\rceil)|x-y|^{2H-2}dxdy\\
\notag&\lesssim_{N,H}&\int_0^N\int_0^{2N}\sigma(\lceil x\rceil)\sigma(\lceil y\rceil)\rho(\lceil x\rceil-\lceil y\rceil)dxdy\\
\notag&&+\int_0^N\int_{2N}^{bu}\sigma(\lceil x\rceil)\sigma(\lceil y\rceil)x^{2H-2}dxdy\\
&\lesssim_{N,\rho,\sigma}&\int_{2N}^{bu}x^{p+2H-2}dx
\le \psi_{p+2H-2}(bu).
\end{eqnarray}
For the second term we have the bound
\begin{align}\label{eq:2018_3}
&\notag\int_0^u\int_0^{bu}\sigma(\lceil x\rceil)\sigma(\lceil y\rceil)\rho(\lceil x\rceil-\lceil y\rceil){\bf 1}\{|x-y|\le N\}dxdy\\
\lesssim_{N,\sigma,\rho}& \int_{2N}^u \int_{y-N}^{y+N} x^p y^p d x dy \lesssim_N \int_N^u y^{2p} d y \le \psi_{2p}(u).
\end{align}
Finally, with $$\eta_N:=\sup_{x\ge N, y\ge N, |x-y|\ge N}\left|\frac{\sigma(\lceil x\rceil) \sigma(\lceil y\rceil) \rho(\lceil x\rceil-\lceil y\rceil)}{x^p y^p|x-y|^{2H-2}}-\kappa\right|,$$
the third term can be estimated as

\begin{eqnarray}\label{eq:2018_4}
\notag&&\int_N^u\int_N^{bu}\sigma(\lceil x\rceil)\sigma(\lceil y\rceil)\rho(\lceil x\rceil-\lceil y\rceil){\bf 1}\{|x-y|> N\}dxdy\\
\notag&\le & (\kappa+\eta_N)\int_{N}^{u}\int_N^{bu} x^p y^p |x-y|^{2H-2}{\bf 1}\{|x-y|>N\}dxdy\\
&\le & (\kappa+\eta_N) f_{p,H}(u,bu)
= (\kappa+\eta_N)u^{2p+2H} f_{p,H}(1,b).
\end{eqnarray}
Since $f_{p,H}(1,b)\gtrsim_{p,H} \psi_{p+2H-2}(b)$ for all $b\ge 1$ by \eqref{eq:f2}, using \eqref{eq:upper_bound_2018} along with \eqref{eq:2018_1},  \eqref{eq:2018_3} and \eqref{eq:2018_4} gives the existence of a finite constant $M(N)=M_{\rho,\sigma,H}(N)$  such that 
\begin{align}\label{eq:bound18} 
\frac{F_{\rho,\sigma}(u,bu)}{ u^{2p+2H}f_{p,H}(1,b)}\le M(N)\frac{\Big\{\psi_{p+2H-2}(bu)+\psi_{2p}(u)\Big\}}{u^{2p+2H}\psi_{p+2H-2}(b)}+(\kappa+\eta_N).
\end{align}
To bound the first term on the right hand side of \eqref{eq:bound18}, we note that 
$$
\psi_{p+2H-2}(bu)\lesssim_{p,H}
\begin{cases} \psi_{p+2H-2}(b)\psi_{p+2H-2}(u)&\text{ when }p+2H-2>-1,\\
\psi_{p+2H-2}(b)+\psi_{p+2H-2}(u)&\text{ when }p+2H-2\le -1.
\end{cases}
$$
Using this, along with the fact that $\max(\psi_{2p}(u),\psi_{p+2H-2}(u))=o(u^{2p+2H})$, as $u\to\infty$, and taking a sup over $b\ge 1$ and let $u\to\infty$ in \eqref{eq:bound18} gives
$$\limsup_{u\to\infty}\sup_{b\ge 1}\frac{F_{\rho,\sigma}(u,bu)}{ u^{2p+2H}f_{p,H}(1,b)}\le \kappa+\eta_N,$$ from which we have the following upper bound on letting $N\to\infty$:
\begin{align}\label{eq:upper_final}
\limsup_{u\to\infty}\sup_{b\ge 1}\frac{F_{\rho,\sigma}(u, bu)}{u^{2p+2H}  f_{p,H}(1,b)}\le \kappa.
\end{align}
For the lower bound, use a similar argument as in the derivation of \eqref{eq:bound18} to get 
\begin{eqnarray}\label{eq:lower_cov}
F_{\rho,\sigma}(u,bu)\notag&\ge  &(\kappa-{\eta}_N)
 \int_{N}^{u}\int_{N}^{bu} x^p y^p |x-y|^{2H-2}{\bf 1}\{|x-y|>N\}dx dy\\
\notag&\ge &(\kappa-{\eta}_N)\Big[f_{p,H}(u,bu)-2f_{p,H}(N,bu)\\
\notag&& -\int_{N}^u\int_{N}^{bu} x^p y^p |x-y|^{2H-2}{\bf 1}\{|x-y|\le N\}dxdy\Big]\\
& \ge& (\kappa-\eta_N)f_{p,H}(u,bu)-\widetilde{M}(N)\frac{\Big\{\psi_{p+2H-2}(bu)+\psi_{2p}(u)\Big\}}{u^{2p+2H}\psi_{p+2H-2}(b)},
\end{eqnarray}
which after dividing by $u^{2p+2H} f_{p,H}(1,b)$, taking an $\inf$ over $b\ge 1$ followed by $u\to\infty$ and $N\to\infty$ gives the lower bound
\begin{align}\label{eq:lower_final}
\liminf_{k\to\infty}\inf_{b\ge 1}\frac{F(u, bu)}{u^{2p+2H}f_{p,H}(1,b)}\ge  \kappa ,
\end{align}
Combining \eqref{eq:lower_final} with \eqref{eq:upper_final}  completes the proof of the lemma.

\end{proof}

\end{document}